\documentclass{article}
\usepackage[paper=a4paper,left=35mm,right=35mm,top=35mm,bottom=35mm]{geometry}

\usepackage{placeins}
\usepackage{graphicx}
\usepackage[sort&compress,numbers]{natbib}
\usepackage{amsmath}
\usepackage{amsthm}
\usepackage{amssymb}
\usepackage{units}
\usepackage{subfigure}
\usepackage{flushend}
\usepackage{lineno}
  \usepackage[affil-it]{authblk}
  
\newcommand*\diff{\mathop{}\!\mathrm{d}}
\newcommand{\R}{\mathbb{R}}
\newcommand{\T}{\top}

\usepackage{color}

\newtheorem{lemma}{Lemma}
 
\newtheorem{proposition}{Proposition}

\title{Controllability analysis and optimal control of biomass drying with reduced order models} 

\author[a]{Marc Oliver Berner} 
\author[b]{Viktor Scherer} 
\author[a]{Martin M\"onnigmann\thanks{Corresponding author. E-mail address: martin.moennigmann@rub.de (M. M\"onnigmann).}} 

\affil[a]{Ruhr-Universit\"at Bochum, Automatic Control and Systems Theory, 
Universit\"atsstr. 150, 44780 Bochum, Germany}
\affil[b]{Ruhr-Universit\"at Bochum, Energy Plant Technology,\\ Universit\"atsstr. 150, 44780 Bochum, Germany} 

\date{}
\begin{document}
\maketitle
\vspace{-.5cm}
\begin{abstract}
Complex industrial processes such as the drying of combustible biomass can be modeled with partial differential equations. Due to their complexity, it is not straightforward to use these models for the analysis of system properties or for solving optimal control problems. We show reduced order models can be derived and used for these purposes for industrial drying processes.
\end{abstract}

\section{Introduction}

Industrial drying processes, such as the drying of wood chips, contribute considerably to the energy consumption of the production of renewable fuels. It is obviously interesting to find energy optimal modes of operation for these processes. 
Dynamic models are useful tools for this task. Whenever the behavior inside the biomass particles needs to be resolved, partial differential equations (PDEs) are required.
In \cite{Sudbrock2015}, for example, the drying of wood chips in rotary dryers is modeled by coupling a discrete element method (DEM) simulation with computational fluid dynamic simulations. The behavior inside the wood particles is considered by a PDE solver that is embedded into the DEM simulation.
Since a direct analysis and control design with infinite-dimensional models is difficult and often not practical for models of industrial processes, it is an option to apply model reduction methods first and to proceed with established methods for finite-dimensional systems.

Reduced order models (ROM) based on proper orthogonal decomposition (POD) and Galerkin projection are suitable for the analysis and optimal control of 
the distributed parameter systems considered here \cite{Moore1981,Jansen2017, Cao2015, Keiper2018, Nagarajan2018, Studinger2013}. 
We showed in \cite{Scherer2016,Berner2017} that models derived with these reduction techniques
can be used to describe
the dynamic drying process inside a single anisotropic wood chip. 
A model reduction of the wood chip model used here was already presented in~\cite{Berner2017}. 
We summarize this reduction as needed for the present paper.
In contrast to~\cite{Berner2017} it is not the purpose of the present paper to describe the model reduction, but to use the resulting reduced model 
for establishing the controllability of the drying process, and to show that the reduced model can be used to find optimal heating time-series by solving an optimal control problem.
We use the empirical framework for nonlinear controllability analysis proposed in~\cite{Hahn2003} and~\cite{Lall1999}, which is based on covariance matrices. More detailed pointers to the literature are given in section~\ref{sec:problemformulation}.

Section~\ref{sec:WoodChipModelling} introduces the wood chip drying process of interest. The control problem and the required nonlinear controllability tools are presented in section~\ref{sec:problemformulation}. We derive a ROM and reduce the computational effort for the controllability analysis in section~\ref{sec:solutionformulation}. The application to the wood chip drying process is presented in section~\ref{sec:ApplicationToWoodChip}. We analyze controllability aspects and the effect of ROM of various orders. Optimal heating profiles are derived with numerical optimal control methods in section~\ref{sec:OCP}. A short conclusion and an outlook can be found in section~\ref{sec:Outlook}.

\section{Modelling of wood chip drying processes}\label{sec:WoodChipModelling}

The drying of biomass in rotary dryers can be modeled by coupling the motion and physical interaction of wood particles inside the drum with the inner particle heat and water diffusion \cite{Sudbrock2014, Sudbrock2015}. The drying process of a single wood chip is characterized by the transient temperature and moisture distribution inside the particle. It must be resolved on the single particle scale due to the size and anisotropy of the wood material \cite{Sudbrock2015,Scherer2016}. A typical size of a wood chip is $\unit[10]{mm}\times\unit[20]{mm}\times\unit[5]{mm}$.

We assume that the drying of a wood chip occurs due to water evaporation at the surface. It depends on the temperature and moisture distribution inside the wood chip and the ambient conditions. 
Let $T(y,t)$ and $x(y,t)$ be the temperature and moisture, respectively, at time $t$ and location $y \in \Omega$, where $\Omega \subset \R^3$ is the volume of the wood chip. Modeling the drying process with Fourier's law of heat conduction and Fick's law of diffusion yields
\begin{subequations}
\label{eqn:PDE}
\begin{align}
\frac{\partial x(y,t)}{\partial t} &= \nabla\Big( {\delta}\big(T(y,t)\big) \nabla x(y,t)\Big) \label{eqn:PDEX}\\
\frac{\partial T(y,t)}{\partial t} &= s^{-1}\big(x(y,t)\big) \nabla \Big( {\lambda}\big(x(y,t)\big) \nabla T(y,t) \Big).\label{eqn:PDET}
\end{align}
\end{subequations}
The material parameters, i.e., the volumetric heat capacity $s(x(y,t))$, the diffusion coefficients $\lambda(x(y,t))$ and $\delta(T(y,t))$ depend on the local temperature or moisture at spatial location $y$ and time $t$. They are stated in appendix A.
Note that $\lambda(x(y, t))\in\R^{3 \times 3}$ and $\delta(T(y, t))\in\R^{3 \times 3}$ due to the anisotropy of the wood. 

The inner particle moisture and temperature distributions are affected by heat and mass fluxes across the particle surface. The boundary conditions for \eqref{eqn:PDE} on the particle surface $\partial \Omega$ with associated normal vector $n$ read

\begin{align}
\begin{aligned}
\label{eqn:PDEBC}
 n^\T \big( \delta(T(y,t))\nabla x(y,t) \big)  \bigg \vert_{\partial \Omega} &= \Gamma_{x}\big(x(y,t),T(y,t)\big)\\
 n^\T \big(\lambda(x(y,t)) \nabla T(y,t) \big)  \bigg \vert_{\partial \Omega}  &= \Gamma_{T}\big(x(y,t),T(y,t)\big) + \alpha T_{\infty}. \end{aligned} 
\end{align}
with
\begin{align}\label{eqn:PDEBCFunctions}
\begin{aligned}
\Gamma_x\big(x,T\big) &= \tfrac{\beta}{\rho_\text{d}} \Big(\rho_\infty - \rho \big( x,T\big) \Big)  \\
\Gamma_T\big(x,T\big) &= -\alpha T + \Delta h_\textrm{ads}\big( x,T\big)  \beta \Big(\rho_\infty - \rho \big( x,T\big) \Big)
\end{aligned} 
\end{align}
\cite[section 2.1]{Scherer2016}, where $x$ and $T$ are short for $x(y, t)$ and $T(y, t)$, respectively, in~\eqref{eqn:PDEBCFunctions}.
The boundary conditions \eqref{eqn:PDEBC} depend on the ambient temperature $T_\infty$, the ambient absolute humidity $\rho_\infty$, the local surface temperature $T(y,t)$, the local absolute humidity on the surface $\rho \big( x(y,t),T(y,t)\big)$, the enthalpy of adsorption $\Delta h_\textrm{ads}\big( x(y,t),T(y,t)\big)$, 
the heat transfer coefficient $\beta$, the mass transfer coefficient $\alpha$, and the density of dry wood $\rho_\text{d}$. 
Note that the boundary conditions are nonlinear, because $\Delta h_\textrm{ads}$ and $\rho$ are nonlinear functions, which we provide in appendix A.

Equations \eqref{eqn:PDE} are solved for initial conditions
\begin{align}
\begin{aligned}
\label{eqn:PDEInitCond}
x(y,t=0)=x_0 \;\; \text{for all }y\\
T(y,t=0)=T_0 \;\; \text{for all }y
\end{aligned}
\end{align}
and boundary conditions \eqref{eqn:PDEBC} with given ambient temperature $T_\infty$ to obtain $x(y,t)$ and $T(y,t)$, i.e., the moisture and temperature distribution inside a wood chip. The initial conditions represent a wet wood chip at room temperature (see Table~\ref{tb:CFDsimulations}).
The total moisture in the wood particle is
\begin{align}
\label{eqn:overallwatercontent}
X(t)=\frac{1}{V}\int_{\Omega}x(y,t)\, \text{d}V
\end{align}
with the wood chip volume $V$ (see Table~\ref{tb:AppxConstParam}).
We do not discuss details of the numerical methods required to solve \eqref{eqn:PDE}-\eqref{eqn:PDEInitCond} but refer to \cite{Sudbrock2014,Sudbrock2015,Scherer2016}, since the present paper focuses on reduced order models and optimal control problems. 

\section{Problem formulation}\label{sec:problemformulation}

We select the ambient temperature $T_\infty$ to be the control input and seek a function $T_\infty(t)$ 
that results in a dry wood chip within a prescribed time span and is at the same time energy optimal in a sense explained below. 
As a preparation, we show that $T_\infty(t)$ permits controlling the temperature and moisture  
by analyzing the controllability of a single wood chip, i.e., the PDEs \eqref{eqn:PDE} subject to the boundary conditions \eqref{eqn:PDEBC}. 

There exist several methods for the controllability analysis of nonlinear distributed parameter systems such as \eqref{eqn:PDE}. 
Some approaches avoid discretizing the PDEs and directly analyze their controllability with semi-group theory \cite{Maidi2015,Delrattre2004}.
Other approaches analyze the finite-dimensional approximation that results for spatial discretization \citep{Leon2002, Hahn2003}. Mature methods \cite{Chen1999, Rosier2009, Levine2009} are available for finite-dimensional systems, but the spatial discretization required for an application to the considered drying process leads to large discretized systems. We will see in Section~\ref{sec:solutionformulation} that order reduction is instrumental to arriving at a finite-dimensional system with an appropriate precision and size.  

A linearization around an operating point is not useful here, since a large temperature range needs to be covered. We therefore perform a nonlinear controllability check with the empirical framework introduced in~\cite{Hahn2002,Hahn2003}. 

\subsection{Empirical controllability Gramian}\label{subsec:CCM}

The empirical controllability analysis is based on simulation results for \eqref{eqn:PDE}. We introduce a discrete model for \eqref{eqn:PDE} that results from spatial discretization for this purpose. Specifically, the wood chip domain $\Omega$ is tessellated with a Cartesian grid consisting of $N$ cubic finite-volume elements of volume $\Delta V$ where the element $i$ belongs to location $y_i\in\Omega$, $i=1,\ldots N$. We obtain
\begin{equation}\begin{aligned}\label{eqn:PDEdiscrete}
\frac{\partial x(y_i,t)}{\partial t} &= \nabla \cdot \Big(\delta \big(T(y_i,t)\big) \nabla x(y_i,t) \Big)\\
\frac{\partial T(y_i,t)}{\partial t} &= s^{-1}(x(y_i,t)) \nabla \cdot \Big( \lambda \big(x(y_i,t)\big) \nabla T(y_i,t) \Big),
\end{aligned}\end{equation}
where $x(y_i,t)$ and $T(y_i,t)$ approximate the moisture $x(y,t)$ and temperature $T(y,t)$ of \eqref{eqn:PDE} at location $y_i$. 
Gradients are approximated in \eqref{eqn:PDEdiscrete}
by balancing heat and mass fluxes through each finite-volume $\Delta V$. The discrete boundary conditions read
\begin{align}
\begin{aligned}
\label{eqn:PDEBCdiscrete}
 n^\T \big( \delta(T(y,t))\nabla x(y,t) \big)  \bigg \vert_{\partial \Omega} &= \Gamma_{x}\big(x(y_i,t),T(y_i,t)\big)\\
 n^\T \big(\lambda(x(y,t)) \nabla T(y,t) \big)  \bigg \vert_{\partial \Omega}  &= \Gamma_{T}\big(x(y_i,t),T(y_i,t)\big) + \alpha T_{\infty}(t). \end{aligned} 
\end{align}
We collect $x(y_i,t)$ and $T(y_i,t)$ for all $i=1,\ldots, N$ in the vector
\begin{align}\label{eqn:StateVectorCFD}
z(t)=[x(y_1,t) \hdots  x(y_N,t)  \,T(y_1,t)  \hdots  T(y_N,t)]^\T ,
\end{align}
$z(t)\in\R^{M}$, with $M=2N$, since $x(y_i, t)\in\R$ and $T(y_i, t)\in\R$. We claim without giving details that a finite-volume model for \eqref{eqn:PDEBCdiscrete} can be written in the form
\begin{align}
\begin{aligned}
\label{eqn:NonlinSysAllg}
\dot{z}(t)&=f\big(z(t)\big)+g\big(z(t)\big)u(t)
\end{aligned}
\end{align}
with $f,g:\R^M\rightarrow\R^M$, state variable $z(t)\in\R^M$ and input $u(t)=T_\infty(t)$, $u(t)\in\R$. 
The original PDEs~\eqref{eqn:PDE} depend on the input $u(t)= T_{\infty}(t)$ through the boundary conditions~\eqref{eqn:PDEBC}. 
The discretized model~\eqref{eqn:PDEdiscrete} inherits the input-affine form of~\eqref{eqn:PDEBC} in the corresponding boundary conditions~\eqref{eqn:PDEBCdiscrete}. Consequently, the finite-dimensional model~\eqref{eqn:NonlinSysAllg} is input-affine, which is a prerequisite for the controllability analysis used here~\cite{Hahn2002,Hahn2003}.
For more details on the finite-volume method we refer to \cite{Moukalled2015, Eymard2000, Fletcher1984} and \cite[pp.~45]{Sudbrock2014}.

The controllability analysis for \eqref{eqn:NonlinSysAllg} is carried out as follows~\cite{Hahn2002,Hahn2003} 
Assume $z(0)$ is a steady state 
\begin{align*}
f\big(z(0)\big)+g\big(z(0)\big)u_{0} = 0
\end{align*}
for some constant input $u_0$. 
We record the response $z_{dli}(t)$ to impulses  
\begin{align}\label{eqn:GramImpulseInput}
u(t)=h_d D_l e_i \delta(t)+u_0
\end{align}
for amplitudes $h_d\in\R$, 
orthonormal matrices $D_l\in\R^{\gamma \times \gamma}$, where $\gamma$ is the number of inputs and $e_i \in \R^\gamma$ are the standard unit vectors.
We can then determine the empirical controllability Gramian
\begin{align}
\label{eqn:CCM}
G=\sum_{i=1}^\gamma \sum_{l=1}^r \sum_{d=1}^s\frac{1}{rsh_d^2}\int_0^\infty \big(z_{dli}(t)-z_{\text{ss},dli}\big)\big(z_{dli}(t)-z_{\text{ss},dli}\big)^\T  \diff t,
\end{align}
where $G \in \R ^{M \times M}$ is symmetric and $z_{\mathrm{ss},dli} = \lim\limits_{t \to \infty}z_{dli}(t)$.  
The Gramian $G$ is composed from data for $s$ input magnitudes $h_d$, $d=1,\ldots ,s$ and $r$ perturbation directions $D_l$, $l=1,\ldots,r$ to account for the nonlinearity in the controllability analysis~\cite{Hahn2002,Hahn2003}. Thus, $s\cdot r \cdot \gamma$ simulations are required in total to determine \eqref{eqn:CCM}.

For nonlinear systems, we cannot make a statement on global controllability, but the following Lemmata are valid locally~\cite{Lall1999}. 

Let $\beta_i$, $i=1,\ldots,M$ refer to the eigenvalues and $v_i$ to the associated eigenvectors of the eigenvalue problem
\begin{align}\label{eqn:EigProblemLarge}
G v_i - \beta_i v_i = 0.
\end{align}

\begin{lemma}\label{lma:LinearCtrbCond} (see, e.g., \cite[Chapter 6.2]{Chen1999}) Assume the system \eqref{eqn:NonlinSysAllg} to be linear and stable. Then \eqref{eqn:NonlinSysAllg} is controllable if and only if $\beta_i>0$ for all $i=1,\ldots,M$, i.e., if and only if the linear controllability Gramian is positive definite.
\end{lemma}

\begin{lemma}\label{lma:CtrlSubspace} (see, e.g., \cite{Moore1981}) Let $\beta_k$ and $v_k$, $k=1,\ldots,M$ be the eigenvalues and associated eigenvectors of \eqref{eqn:CCM} for a stable linear system. Then all points in the state space that can be reached from the origin within a prescribed time $t$ with an energy $\int_0^t u^\T(\tau) u(\tau) \diff\tau \leq 1$ are located within a hyperellipsoid with semi axes $\sqrt{\beta_k} v_k$, $k=1,\ldots,M$.
\end{lemma}

Essentially, the eigenvalues $\beta_1 \geq \ldots \geq \beta_M $ and their corresponding eigenvectors $v_1, \ldots , v_M $ determine the range and direction in which the system is easiest to control.

It is impractical to determine Gramians~\eqref{eqn:CCM} with the discretized PDEs~\eqref{eqn:NonlinSysAllg} if $s\cdot r\cdot \gamma$ is large. 
Even though the particular optimal control problem solved in section~\ref{sec:ApplicationToWoodChip} only involves a single input ($\gamma= 1$), computing~\eqref{eqn:CCM} with~\eqref{eqn:NonlinSysAllg} is already too time-consuming.\footnote{Calculating the Gramian~\eqref{eqn:CtrbGramROMDiscrete} with the reduced model in Section~\ref{subsec:CTRLDryingProcess} requires \unit[188]{s} with a matlab implementation on a standard desktop PC with an Intel i7-6700 CPU running at 3.4GHz. The corresponding calculation with the discretized PDEs~\eqref{eqn:NonlinSysAllg} was incomplete after one day.}
A reduction of the model~\eqref{eqn:PDEdiscrete} is thus instrumental to performing the controllability analysis. We introduce a method in the next section that results in both an acceleration of the simulations and a reduction of the eigenvalue problem.
Note that we state and treat the problem for arbitrary $s$, $r$ and $\gamma$ for the sake of generality.  

\section{Solution formulation}\label{sec:solutionformulation}
A ROM is derived in section~\ref{subsec:modelreduction} and used to reduce the computational effort for the controllability analysis in~\ref{subsec:CtrbGramROM}.

\subsection{Reduced order model}\label{subsec:modelreduction}
We briefly introduce the model reduction procedure as required for the present paper and refer to \cite{Berner2017, Scherer2016} for details. The model reduction is based on POD and subsequent Galerkin projection \cite{Sirovich1987}.
We discuss the reduction of Fourier's law of heat conduction~\eqref{eqn:PDET}. Fick's law of diffusion~\eqref{eqn:PDEX} can be treated analogously. 

It is the first step to obtain so called snapshots
\begin{align*}
z_T(t_j)&=[T(y_1,t_j) \hdots  T(y_N,t_j)]^\T
\end{align*}
$z_T(t_j)\in\R^{N}$ that solve or approximately solve \eqref{eqn:PDET} at time points $t_j$, $j=1,\ldots,m$ for boundary conditions \eqref{eqn:PDEBC} and given initial conditions \eqref{eqn:PDEInitCond} at the spatial points $y_i\in\Omega$, $i=1,\ldots,N$. 
 
Assuming that $b$ linear independent snapshots exist, we can find $b$ orthonormal basis vectors  $\phi_{T,k}=[\varphi_{T,k}(y_1) \hdots  \varphi_{T,k}(y_N) ]^\T $, $\phi_{T,k}\in\R^{N}$, $k=1,\ldots,b$, of the snapshot set, also called modes, such that 
\begin{align}\label{eqn:SumLinComb}
T(y_i,t_j)&= \bar{T}(y_i) + \textstyle\sum_{k=1}^{b} c_{T,k}(t_j) \varphi_{T,k}(y_i) 
\end{align} 
where 
\begin{align}\label{eqn:SnpMean}
\bar{T}(y_i) = \tfrac{1}{m}\textstyle\sum_{j=1}^{m} T(y_i,t_j),
\end{align} 
$\bar{T}(y_i)\in\R$, is the time average and
\begin{align}\label{eqn:TimeCoeff}
c_{T,k}(t_j)=\langle  T(y_i,t_j)- \bar{T}(y_i),\, \varphi_{T,k}(y_i) \rangle,
\end{align} 
$c_{T,k}(t_j) \in\R$, are time-dependent coefficients. The brackets $\langle \cdot, \, \cdot \rangle$ denote the standard inner product in its discrete form
\begin{align}\label{eqn:InnerProductL2}
\langle a(\cdot), b(\cdot)\rangle &=\textstyle\sum_{i=1}^{N} a(y_i)\,b(y_i) \Delta V,
\end{align}
for $a(y_i),b(y_i):\, \Omega \rightarrow \R$ and the discrete volume $\Delta V \in \R$. Truncating the sum \eqref{eqn:SumLinComb} at some cut-off value $n_T<b$ does not result in an exact representation but in an approximation of the initial set of snapshots. A systematic method to determine the modes and number $n_T$ so that the truncated sum results in a good approximation is a singular value decomposition of the snapshot set. We refer to \cite{Sirovich1987,Cordier2008a,Cordier2008b} for further details. 
Since $n_T$ corresponds to the number of ODEs in the ROM, $n_T$ should be chosen as small as possible. The approximation reads
\begin{align}\label{eqn:SumLinApprox}
T(y_i,t_j)&\approx \bar{T}(y_i) + \textstyle\sum_{k=1}^{n_T} \varphi_{T,k}(y_i) c_{T,k}(t_j).
\end{align}
We now seek $n_T$ ordinary differential equations for the coefficients \eqref{eqn:TimeCoeff} such that their time continuous results $c_{T,k}(t)$ yield a reasonable approximation for \eqref{eqn:SumLinApprox} at $t=t_j$ and all times in between those sample times. 
We apply three simplifications in the explanation to follow:  
(i) We assume a continuous representation of $\varphi_{T,k}(y_i)$, i.e., we assume that $\varphi_{T,k}(y)$ is defined for all points $y\in\Omega$, since it allows us to apply integrals and differential operators. 
(ii) We assume the material parameters $s$ and $\lambda$ to be constant in order to avoid tedious applications of the product and chain rules. We stress this assumption is only applied to simplify the summary of the method. The model reductions in Section~\ref{sec:ApplicationToWoodChip} are performed with the non-constant quantities $s(x(y, t))$, $\lambda(x(y,t))$ and $\delta(T(y,t))$ given in appendix A and all results presented in Section~\ref{sec:OCP} are obtained with these dependencies. (iii) We omit the dependence on $y_i$ and $t_j$ for brevity.
Substituting \eqref{eqn:SumLinApprox} into \eqref{eqn:PDET} yields
\begin{align*}
&\tfrac{\partial}{\partial t}\big(\bar{T}+\textstyle\sum_{k=1}^{n_T} \varphi_{T,k}c_{T,k}\big)\approx s^{-1} \nabla \cdot \Big( {\lambda} \nabla \big(\bar{T}+\textstyle\sum_{k=1}^{n_T} \varphi_{T,k}c_{T,k}\big) \Big)
\end{align*}
The projection onto the first $l=1,\ldots,n_T$ modes reads
\begin{align}\label{eqn:PDEProjection}
\begin{aligned} &\big\langle \tfrac{\partial}{\partial t}\textstyle\sum_{k=1}^{n_T} \varphi_{T,k} c_{T,k}, \; \varphi_{T,l} \big\rangle
\approx \\
&\Big\langle  s^{-1} \nabla \cdot  \Big( {\lambda} \nabla \big(\bar{T}+\textstyle\sum_{k=1}^{n_T} \varphi_{T,k}c_{T,k}\big) \Big) ,\;   \varphi_{T,l} 
\Big\rangle. \end{aligned} 
\end{align}
Exploiting the time independence and orthonormality of the modes, i.e., 
\begin{align}\label{eqn:OrthoNormality}
\langle \varphi_{x,l} ,\, \varphi_{x,k} \rangle=\delta_{l,k}
\end{align}
with Kronecker's delta $\delta_{l,k}$, results in the desired ordinary differential equations
\begin{align}\label{eqn:ROM}
\dot{c}_{T,l} \approx \Big\langle  s^{-1} \nabla \cdot  \Big( {\lambda} \nabla \big(\bar{T}+\textstyle\sum_{k=1}^{n_T} \varphi_{T,k}c_{T,k}\big) \Big) ,\;   \varphi_{T,l} 
\Big\rangle.
\end{align}
The $l=1,\ldots,n_T$ ODEs \eqref{eqn:ROM} constitute the ROM for temperature diffusion. Note that the only time-dependent variables are the coefficients $c_{T,k}$, $k=1,\ldots,n_T$. 

We apply Gauss's theorem to explicitly consider the boundary conditions and the control input in \eqref{eqn:ROM}. Since the boundary conditions \eqref{eqn:PDEBC} are functions of temperature and moisture, we need both, the temperature approximation \eqref{eqn:SumLinApprox} and the corresponding moisture approximation. Let $\bar{x}(y_i)$, $n_x$, $\varphi_{x,k}(y_i) $ and $c_{x,k}(t_j)$, $i=1,\ldots,N$, $j=1,\ldots,m$, $k=1,\ldots,n_x$ be the time average, cut-off value, modes and time coefficients, respectively, obtained from a set of snapshots for the moisture determined with the methods presented in section \ref{subsec:modelreduction}. Then
\begin{align}\label{eqn:SumLinApproxWater}
x(y_i,t_j)&\approx \bar{x}(y_i) + \textstyle\sum_{k=1}^{n_x} \varphi_{x,k}(y_i) c_{x,k}(t_j)
\end{align}
is an approximation like \eqref{eqn:SumLinApprox} but determined for the moisture. Without giving details we state that \eqref{eqn:ROM} is transformed into
\begin{align}\label{eqn:ROMGauss}
\dot{c}_{T,l} &\approx -\textstyle\int_\Omega  \Big(\lambda \nabla \big( \bar{T}+\textstyle\sum_{k=1}^{n_T}\varphi_{T,k}c_{T,k}\big)\Big) \cdot \nabla \varphi_{T,l} \diff V + \nonumber\\
& \textstyle\int_{\partial \Omega} \varphi_{T,l} \, s^{-1}\, \Gamma_{T} \big(\bar{x}+\textstyle\sum_{k=1}^{n_x}\varphi_{x,k}c_{x,k}, \\ &\bar{T}+\textstyle\sum_{k=1}^{n_T}\varphi_{T,k}c_{T,k} \big) \diff S  + 
T_{\infty}(t) \,  \textstyle\int_{\partial \Omega} \varphi_{T,l} \, s^{-1}\, \alpha \diff S \nonumber
\end{align}
when the volume integral of the inner product is transformed into a surface integral with Gauss's theorem. The boundary condition $\Gamma_T$ and ambient temperature $T_\infty(t)$ appear explicitly in \eqref{eqn:ROMGauss} (see \cite{Berner2017} for details). Note that the ODEs \eqref{eqn:ROMGauss} are nonlinear due to the nonlinearity of the boundary condition $\Gamma_T$. In fact, \eqref{eqn:ROMGauss} is input-affine when the ambient temperature $T_\infty(t)$ is considered to be the control input. Note that this is a prerequisite for the calculation of \eqref{eqn:CCM} according to \cite{Hahn2002,Hahn2003}.

The initial conditions for the temperature in \eqref{eqn:PDEInitCond} are considered by projecting \eqref{eqn:PDEInitCond} onto the first $k=1,\ldots,n_T$ modes. If the temperature in \eqref{eqn:PDEInitCond} is part of the snapshot set then the coefficients $c_{T,k}(t_0=0)$, $k=1,\ldots,n_T$ from decomposition \eqref{eqn:SumLinApprox} for $t_0=0$ are the desired initial conditions for \eqref{eqn:ROMGauss}. 

We repeat the procedure of section \ref{subsec:modelreduction} with the moisture approximation \eqref{eqn:SumLinApproxWater} to derive $n_x$ ODEs for the moisture diffusion. This yields
 \begin{align}\label{eqn:PDEProjectionX}
  \dot{c}_{x,l}(t)&= -\int_\Omega  \Big(\delta \nabla \big( \bar{x}+\textstyle\sum_{k=1}^{n_x}\varphi_{x,k}c_{x,k}\big)\Big) \cdot \nabla \varphi_{x,l} \diff V + \nonumber\\
& \int_{\partial \Omega} \varphi_{x,l} \, \Gamma_{x} \big(\bar{x}+\textstyle\sum_{k=1}^{n_x}\varphi_{x,k}c_{x,k}, \, \bar{T}+\textstyle\sum_{k=1}^{n_T}\varphi_{T,k}c_{T,k} \big) \diff S,
\end{align}
where $l= 1, \dots, n_x$. 
The set of $n_x+n_T=n$ ODEs 
 \begin{align}\label{eqn:PODGalModel}
\dot{c}(t) &= \begin{bmatrix}
f_{\text{ROM,x},1}\big(c_{x,k}(t),c_{T,l}(t)\big)\\
\vdots\\
f_{\text{ROM,x},n_x}\big(c_{x,k}(t),c_{T,l}(t)\big)\\
f_{\text{ROM,T},1}\big(c_{x,k}(t),c_{T,l}(t)\big)\\
\vdots\\
f_{\text{ROM,T},n_T}\big(c_{x,k}(t),c_{T,l}(t)\big)\\
\end{bmatrix} =f_\text{ROM}\big(c(t) \big),
\end{align}
where $f_{\text{ROM,T}}$ and $f_{\text{ROM,x}}$ refer to the r.h.s.\ of~\eqref{eqn:PDEProjection} and~\eqref{eqn:PDEProjectionX}, respectively, constitute the ROM with $f_\mathrm{ROM}:\R^n\rightarrow\R^n$ and 
\begin{align}\label{eqn:PODGalStates}
c(t)&=[c_{x,1}(t) \hdots  c_{x,n_x}(t) \, c_{T,1}(t)  \hdots  c_{T,n_T}(t)]^\T  \;\in\R^{n}.
\end{align}
Note that all ODEs are coupled, since the states \eqref{eqn:PODGalStates} appear in all ODEs. Solving \eqref{eqn:PODGalModel} for given initial conditions yields time series for $c_{T,i}(t)$ and $c_{x,i}(t)$ that are substituted in \eqref{eqn:SumLinApprox} and \eqref{eqn:SumLinApproxWater} to determine the temperature and moisture. Collecting \eqref{eqn:SumLinApprox} and \eqref{eqn:SumLinApproxWater} as in 
\eqref{eqn:StateVectorCFD} yields the state variable of the finite-volume model 
\begin{align}\label{eqn:StateVectorCFDApprox}
z(t) \approx \Phi c(t) + \bar{z},
\end{align}
where the modes of~\eqref{eqn:SumLinApprox} and \eqref{eqn:SumLinApproxWater} are collected in
\begin{align}\label{eqn:AllModes}
\Phi = \begin{bmatrix}\begin{smallmatrix} 
\varphi_{x,1}(y_1)&\ldots&\varphi_{x,n_x}(y_1)&0&\ldots&0\\
\ldots&\ldots&\ldots&\ldots&\ldots&\ldots\\
\varphi_{x,1}(y_N)&\ldots&\varphi_{x,n_x}(y_N)&0&\ldots&0\\
0&\ldots&0&\varphi_{T,1}(y_1)&\ldots&\varphi_{T,n_T}(y_1)\\
\ldots&\ldots&\ldots&\ldots&\ldots&\ldots\\
0&\ldots&0&\varphi_{T,1}(y_N)&\ldots&\varphi_{T,n_T}(y_N)\\
\end{smallmatrix}\end{bmatrix},
\end{align}
$\Phi\in\R^{M\times n}$, and 
\begin{align}\label{eqn:AllMeans}
\bar{z}=[\bar{x}(y_1) \hdots  \bar{x}(y_N)  \, \bar{T}(y_1)  \hdots  \bar{T}(y_N)]^\T\; \in\R^M
\end{align}
is the time average of the snapshot set~\eqref{eqn:SnpMean} for temperature and moisture. 

\subsection{ROM based controllability Gramian}\label{subsec:CtrbGramROM}
The ROM of section \ref{subsec:modelreduction} is ultimately used to solve the eigenvalue problem \eqref{eqn:EigProblemLarge}. 
The required steps are summarized in propositions~\ref{Prp:ApproxGram} and~\ref{Prp:SmallEigProblem} below. 
We determine the impulse response $z_{dli}(t)$ required for the Gramian \eqref{eqn:CCM} with the reduced model \eqref{eqn:PODGalModel}. 
More precisely, the input \eqref{eqn:GramImpulseInput} is applied to the ROM \eqref{eqn:PODGalModel} to determine $c^{dli}(t)$, i.e., the impulse response of the ROM first. The desired impulse response $z_{dli}(t)$ of the finite-volume model then results by mapping the ROM state variables to the original state variables with \eqref{eqn:StateVectorCFDApprox}.

\begin{proposition}\label{Prp:ApproxGram}
Let
\begin{align}\label{eqn:CtrbGramROM}
W=\sum_{i=1}^\gamma \sum_{l=1}^r \sum_{d=1}^s \frac{1}{rsh_d^2}\int_0^\infty \big(c^{dli}(t)-c_{\text{ss}}^{dli}\big)\big(c^{dli}(t)-c_{\text{ss}}^{dli}\big)^\T  \diff t,
\end{align}
$W\in\R^{n\times n}$, refer to the controllability Gramian of the reduced order model 
where $c_{\text{ss}}^{dli} = \lim\limits_{t \to \infty}c^{dli}(t)$. 
Then the Gramian \eqref{eqn:CCM} can be approximated by
\begin{align}
G\approx\tilde{G}=\Phi W \Phi^\T. \label{eqn:ApproxCtrbGram}
\end{align}
\end{proposition}

\begin{proof} Substituting \eqref{eqn:StateVectorCFDApprox} into \eqref{eqn:CCM} yields
\begin{align*}
G &\approx \sum_{i=1}^\gamma \sum_{l=1}^r \sum_{d=1}^s \frac{1}{rsh_d^2}\int_0^\infty \big(\Phi c^{dli}(t) + \bar{z} -\Phi c_{\text{ss}}^{dli} - \bar{z} \big) \\ &\hspace{1cm} \big(\Phi c^{dli}(t) + \bar{z} -\Phi c_{\text{ss}}^{dli} - \bar{z} \big)^\T  \diff t\\
&= \Phi \sum_{i=1}^\gamma \sum_{l=1}^r \sum_{d=1}^s  \frac{1}{rsh_d^2}\int_0^\infty \big( c^{dli}(t)  -c_{\text{ss}}^{dli} \big)\big(c^{dli}(t) - c_{\text{ss}}^{dli}  \big)^\T  \diff t \;  \Phi^\T\\
&=\Phi W \Phi^\T=\tilde{G}.
\end{align*}
which is the claim~\eqref{eqn:ApproxCtrbGram}.
\end{proof}

Since $\tilde{G}$ is an approximation for Gramian~\eqref{eqn:CCM}, 
\begin{align}\label{eqn:EigProblemLargeApprox}
\tilde{G} \tilde{v}_k - \tilde{\beta}_k \tilde{v}_k = 0
\end{align}
is an approximation for eigenvalue problem \eqref{eqn:EigProblemLarge}, where
\begin{align}\label{eqn:LinApproxEig}
\beta_k \approx \tilde{\beta}_k \;\text{and}\;
v_k \approx \tilde{v}_k 
\end{align}
are approximations for $k=1,\ldots,n$ eigenvalues and corresponding eigenvectors of $G$, respectively.

\begin{proposition}\label{Prp:SmallEigProblem}
Let $\tilde{\beta}_k$ be as in~\eqref{eqn:EigProblemLargeApprox}. 
Then the non-zero eigenvalues of \eqref{eqn:EigProblemLargeApprox} are equal to those of the smaller $n$-dimensional eigenvalue problem
\begin{align}\label{eqn:EigProblemSmall}
W \Phi^\T \Phi w_k  -  \tilde{\beta}_k w_k = 0
\end{align}
and the respective eigenvectors of \eqref{eqn:EigProblemLargeApprox} are given by 
\begin{align}\label{eqn:EigVectorLargeFromSmall}
\tilde{v}_k =\Phi w_k,
\end{align}
where $w_k\in\R^n$ is the eigenvector of \eqref{eqn:EigProblemSmall}.
\end{proposition}

\begin{proof}We first consider the eigenvalues. Substituting \eqref{eqn:ApproxCtrbGram} into \eqref{eqn:EigProblemLargeApprox} and using Sylvester's determinant identity, we can write the characteristic polynomial determinant for \eqref{eqn:EigProblemLargeApprox} as
\begin{align}\label{eqn:CharPolyEigProb}
\det\big( \tilde{\beta}_k I_M -\Phi W\Phi^T \big) = \tilde{\beta}_k^{M-n} \det\big( \tilde{\beta}_k I_n - W\Phi^T\Phi\big),
\end{align}
where $I_M$ and $I_n$ are the $M\times M$ and $n\times n$ identity matrices, respectively. 
We observe that the non-trivial roots $\tilde{\beta}_k$ of the right hand side of \eqref{eqn:CharPolyEigProb}, i.e., the eigenvalues of~\eqref{eqn:EigProblemSmall}, correspond to the non-zero roots of the left hand side of \eqref{eqn:CharPolyEigProb}, i.e., the non-zero eigenvalues of \eqref{eqn:EigProblemLargeApprox}.

Now consider the eigenvectors. Left multiplying \eqref{eqn:EigProblemSmall} with $\Phi$ yields
\begin{align}\label{eqn:EigProblemLargeApprox2}
\Phi W \Phi^\T \Phi w_k =\tilde{\beta}_k \Phi w_k.
\end{align}
Substituting $\Phi w_k$ in \eqref{eqn:EigProblemLargeApprox2} by \eqref{eqn:EigVectorLargeFromSmall} yields \eqref{eqn:EigProblemLargeApprox}. 
\end{proof}

Note that \eqref{eqn:InnerProductL2} and \eqref{eqn:OrthoNormality} imply $\Phi^\T\Phi= \text{diag}(\nicefrac{1}{\Delta V},\ldots,\nicefrac{1}{\Delta V})$.

\section{Application to the drying process of wood chips}\label{sec:ApplicationToWoodChip}

We apply the model reduction procedure presented in section~\ref{subsec:modelreduction} to the drying problem introduced in section~\ref{sec:WoodChipModelling}. We evaluate the ROM in section~\ref{subsec:ROMevaluation} and analyze the controllability in section~\ref{subsec:CTRLDryingProcess}. The influence of the degree of reduction is addressed in section~\ref{subsec:ComparisonOfROMs}. 

\subsection{Reduced order model evaluation}\label{subsec:ROMevaluation}

\begin{table}[t]
\begin{center}
\caption{Simulation conditions for the drying process of  wood chips}\label{tb:CFDsimulations}
\begin{tabular}{llr}
initial wood chip moisture & $x(t=0)$ & \unit[0.8]{kg/kg}\\ 
initial wood chip temperature & $T(t=0)$ & \unit[298.15]{K}\\
simulation duration & &  \unit[1100]{s}\\
number of grid points & $N$ &  \unit[1000]{}\\
number of snapshots & $m$ &  \unit[100]{}\\
\textbf{case A:} &&\\
ambient temperature & $T_\infty(t<0)$ & \unit[298.15]{K}\\
& $T_\infty(t\geq0)$ & \unit[373.15]{K}\\
\textbf{case B:} &&\\
ambient temperature & $T_\infty(t<0)$ & \unit[298.15]{K}\\
& $T_\infty(t\geq0)$ & \unit[335.65]{K}\\
\end{tabular}
\end{center}
\end{table}

We determine snapshots for the temperature and moisture from a simulation of \eqref{eqn:PDE} for the conditions stated in table \ref{tb:CFDsimulations}, case A. These conditions represent a typical drying process where an initially wet wood chip at room temperature is exposed to hot dry air until a steady state is reached after approximately $\unit[1100]{s}$. 
The drying process can be modeled by applying a step function to the ambient temperature $T_\infty(t)$ with $T_\infty(t<0)=\unit[298.15]{K}$ and final temperature $T_\infty(t\geq 0)=\unit[373.15]{K}$. 

We determine the modes $\varphi_{x,l}(y_i)$, $\varphi_{T,k}(y_i)$ and coefficients $c_{x,l}(t)$, $ c_{T,k}(t)$, $l=1,\ldots,n_x$, $k=1,\ldots,n_T$  so that \eqref{eqn:SumLinApprox} and \eqref{eqn:SumLinApproxWater} yield approximations for the temperature and moisture, respectively. 
We select an appropriate order $n$ of the reduced model by analyzing the approximation error for the total moisture $X$ introduced in~\eqref{eqn:overallwatercontent}. 
We use the total moisture $X$ for this purpose, since the optimal control problem for the drying process treated in section~\ref{sec:OCP} requires a terminal constraint on $X$. Specifically, we determine the normalized root-mean-square error (NRMSE) for the total moisture
\begin{align}\label{eqn:MeanRelError}
\varepsilon(n)=\frac{\sqrt{ \frac{1}{m} \sum_{j=1}^{m}\Big(X(t_j)-\frac{1}{N}\textstyle\sum_{i=1}^N\big( \sum_{l=1}^n \varphi_{x,l}(y_i)c_{x,l}(t_j)+\bar{x} (y_i) \big) \Big)^2} }
{\max_{j} X(t_j) - \min_{j} X(t_j)}
\end{align}
with respect to simulation results for the finite-volume model~\eqref{eqn:NonlinSysAllg}. 
The error \eqref{eqn:MeanRelError} is shown in Figure~\ref{fig:ErrorTotalXforModes} as a function of the cut-off value $n$. As expected, the error decreases with increasing number of modes from  $\varepsilon(6)=0.02\%$ to $\varepsilon(50)=8.8\times10^{-11}\%$. We consider only orders $n\ge 6$ here and in the remainder of the paper, because the integration of the ROM for order $n=4$ was unstable. 
We anticipate we choose $n_x=n_T=3$, thus $n=6$, after showing that the corresponding reduced model is not only sufficiently accurate, but also has the required controllability properties in section~\ref{subsec:CTRLDryingProcess}. 

\begin{figure}[t]
        \centering
    \includegraphics[width=0.7\textwidth]{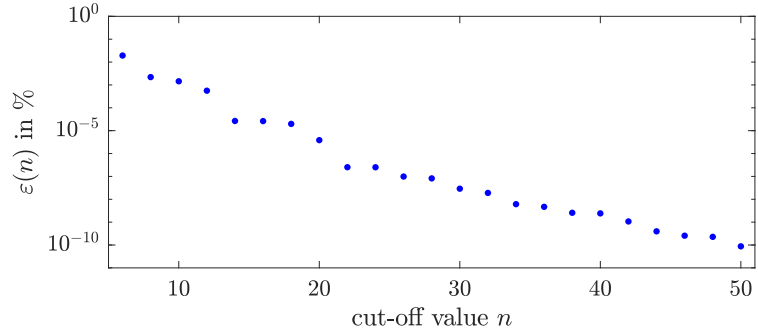}
               \caption{Normalized-root-mean square error over all times for the total moisture $X(t)$ for cut-off values $n=6,8,10,\ldots,50$. Note the semi-logarithmic scale.}
        \label{fig:ErrorTotalXforModes}
\end{figure}

\begin{figure}[b!]
        \centering
\subfigure{\includegraphics[width=0.35\textwidth]{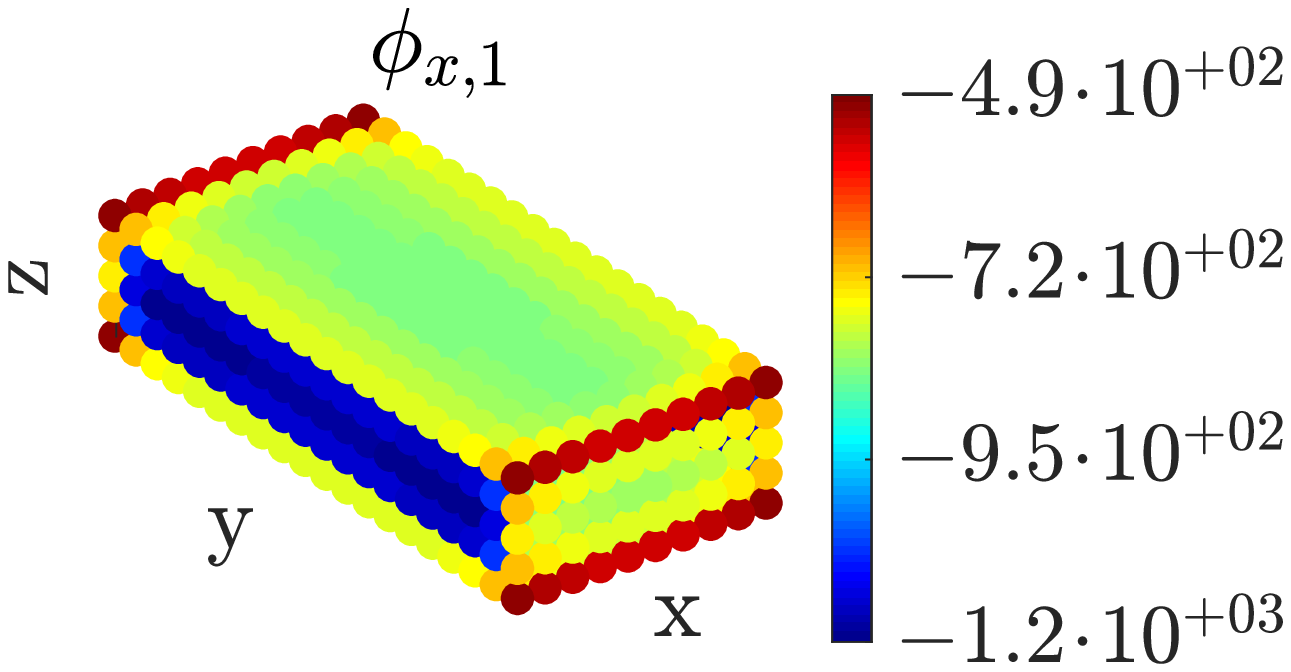}}
\subfigure{\includegraphics[width=0.35\textwidth]{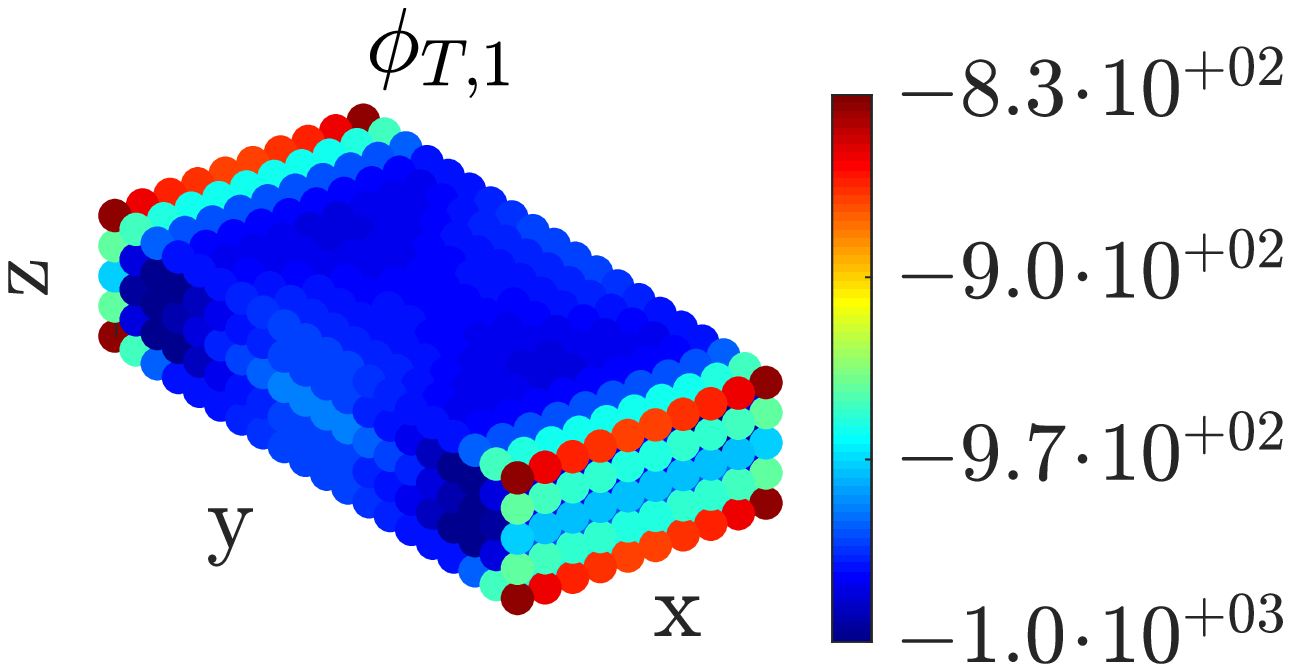}}
        \caption{First mode for the moisture $\phi_{x,1}$ (left) and temperature $\phi_{T,1}$ (right) obtained by performing a POD according to section \ref{subsec:modelreduction}. }
        \label{fig:PODModes}
\end{figure}

The first modes $\phi_{x,1}$ and $\phi_{T,1}$ and the coefficients $c_{x,l}(t_j)$, $c_{T,k}(t_j)$, $k=1,\ldots,3$, are shown in Figures \ref{fig:PODModes} and \ref{fig:PODCoeff} (red crosses), respectively.  
We stress that all simplifications that were used for explanatory reasons in section \ref{subsec:modelreduction} do not apply here in chapter \ref{sec:ApplicationToWoodChip}. Specifically, the material parameters $s$, $\lambda$ and $\delta$ are nonlinear functions of the local moisture or temperature approximations and the moisture and heat diffusion coefficients $\lambda$ and $\delta$ are of dimension $\R^{3 \times 3}$ (see appendix A).

\begin{figure}[tb]
        \centering
\subfigure{\includegraphics[width=0.35\textwidth]{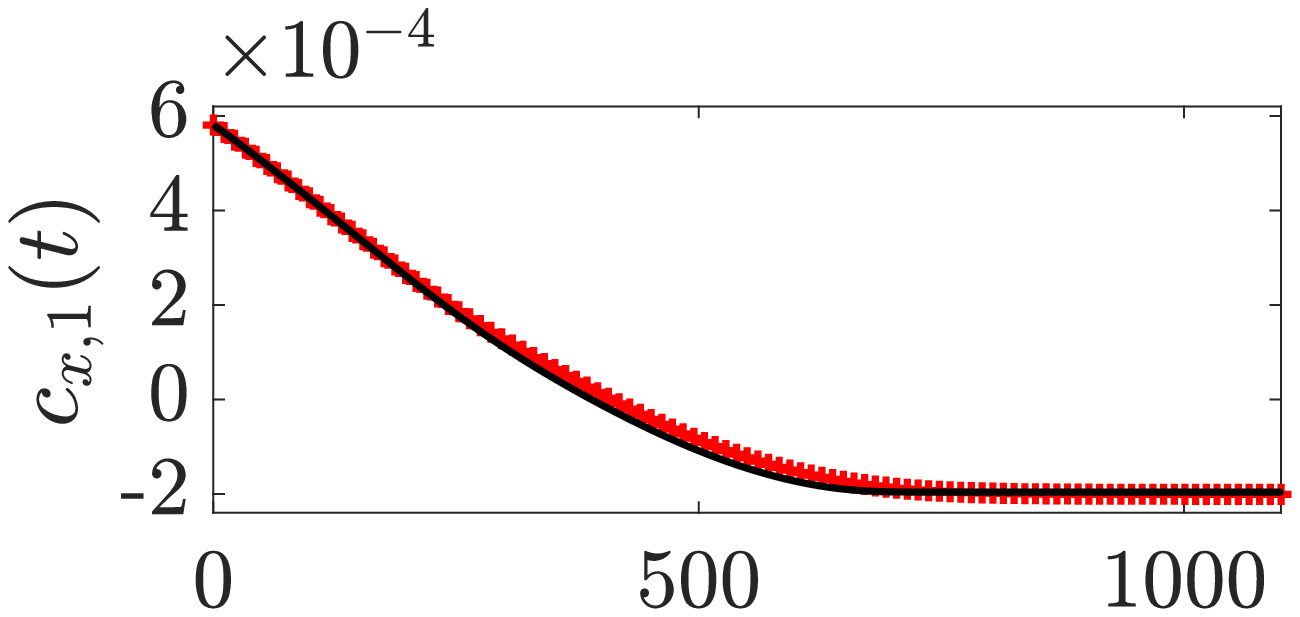}}
\subfigure{\includegraphics[width=0.35\textwidth]{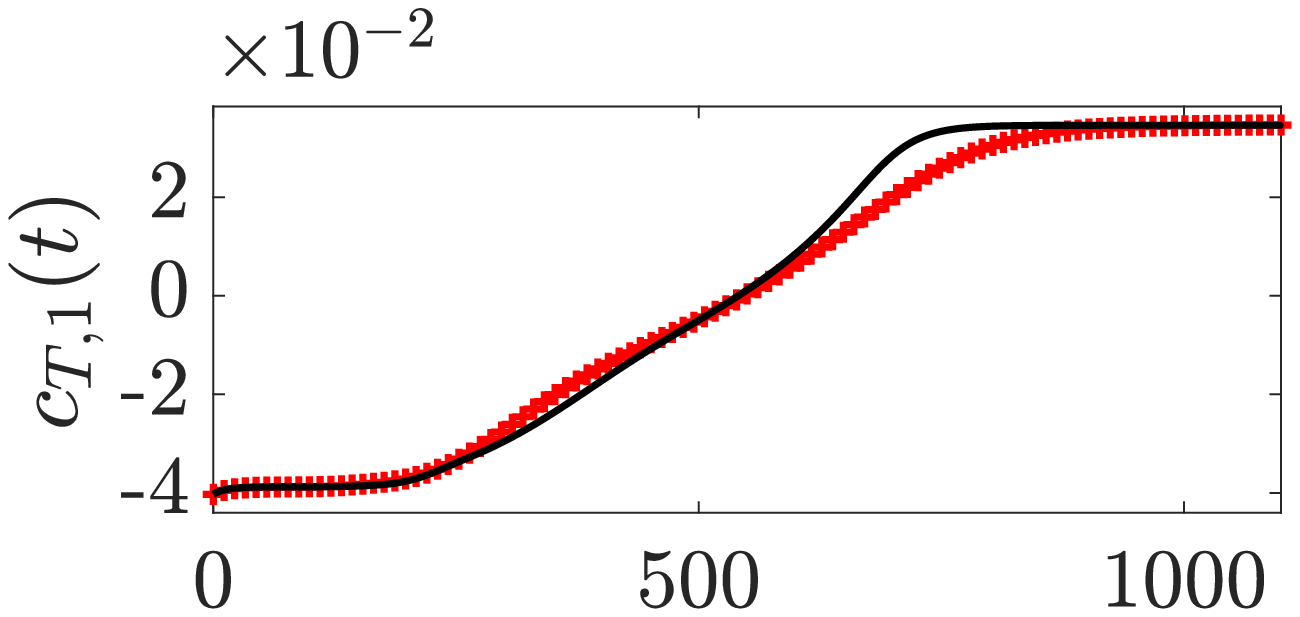}}\\
\subfigure{\includegraphics[width=0.35\textwidth]{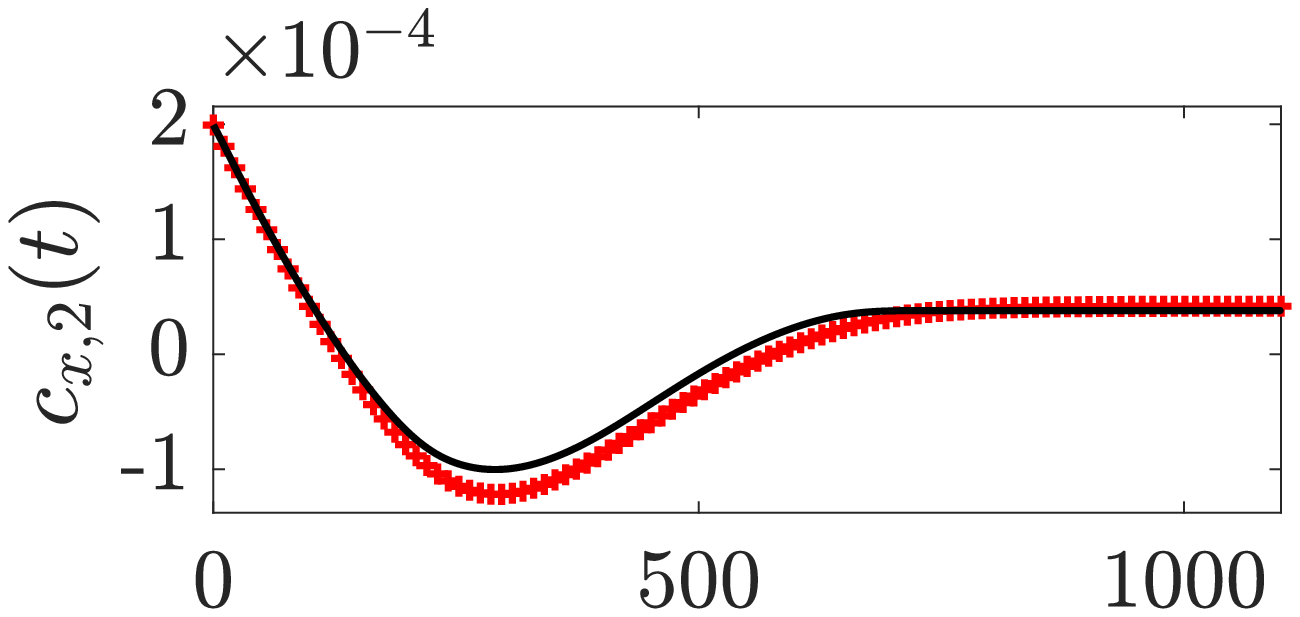}}
\subfigure{\includegraphics[width=0.35\textwidth]{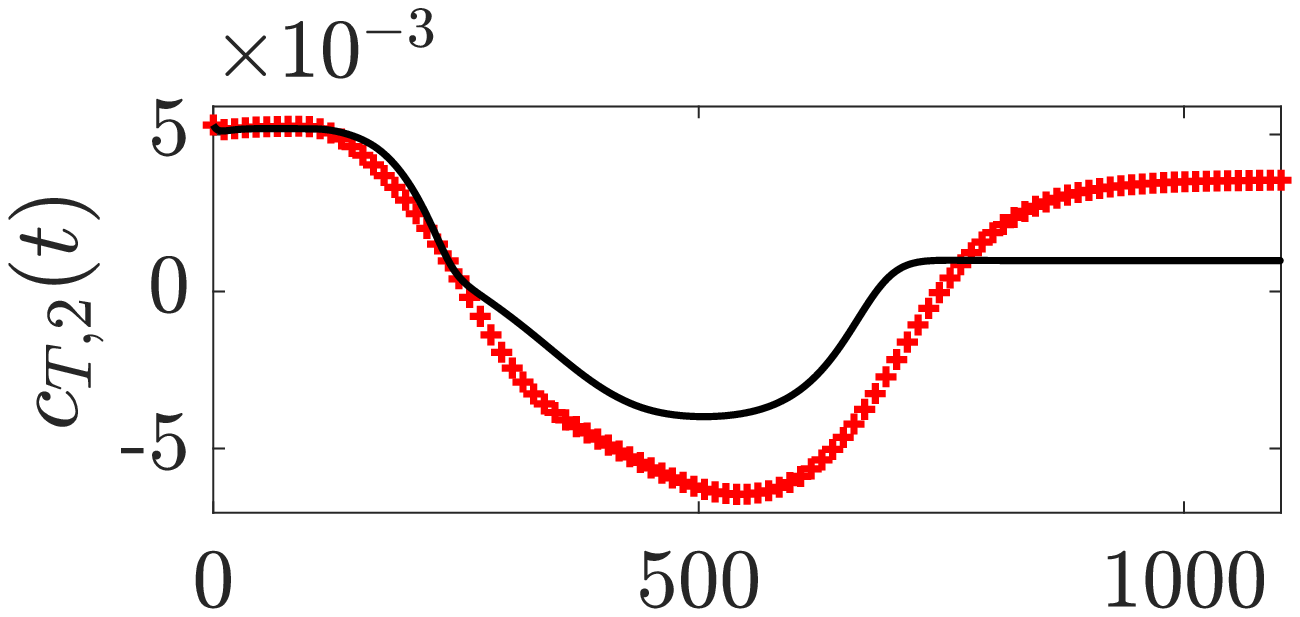}}\\
\subfigure{\includegraphics[width=0.35\textwidth]{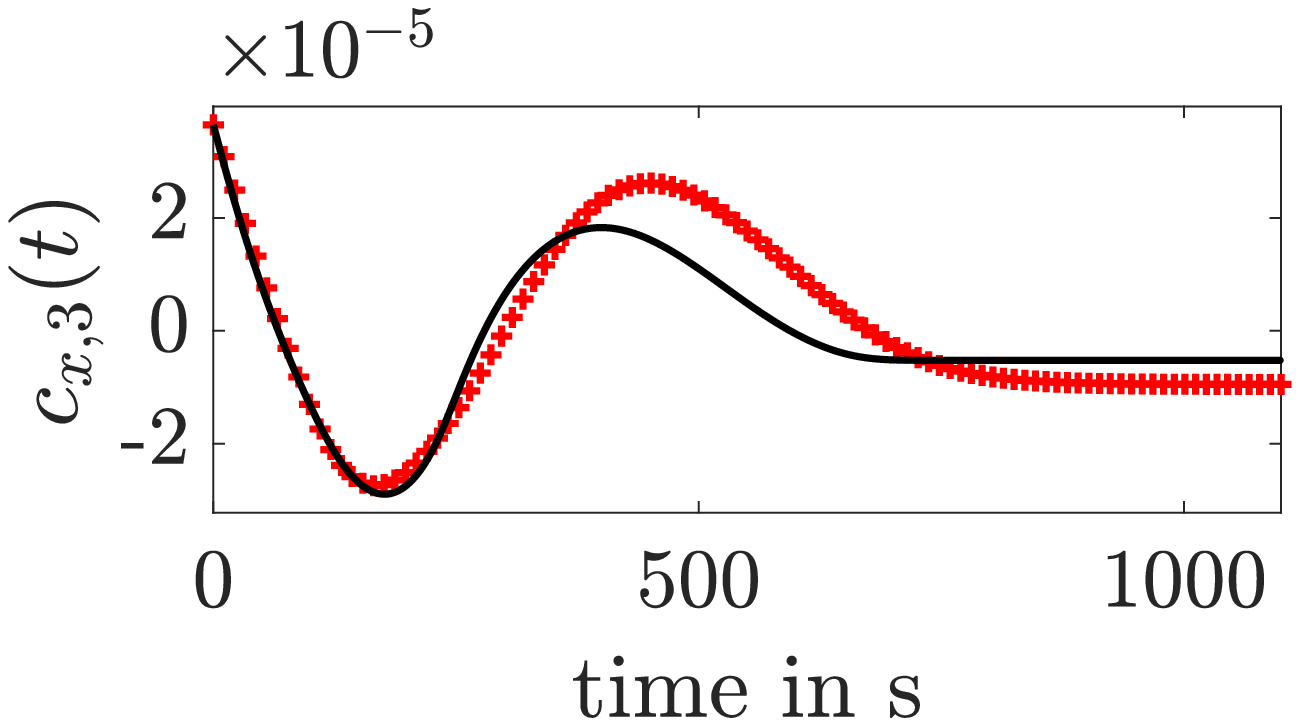}}
\subfigure{\includegraphics[width=0.35\textwidth]{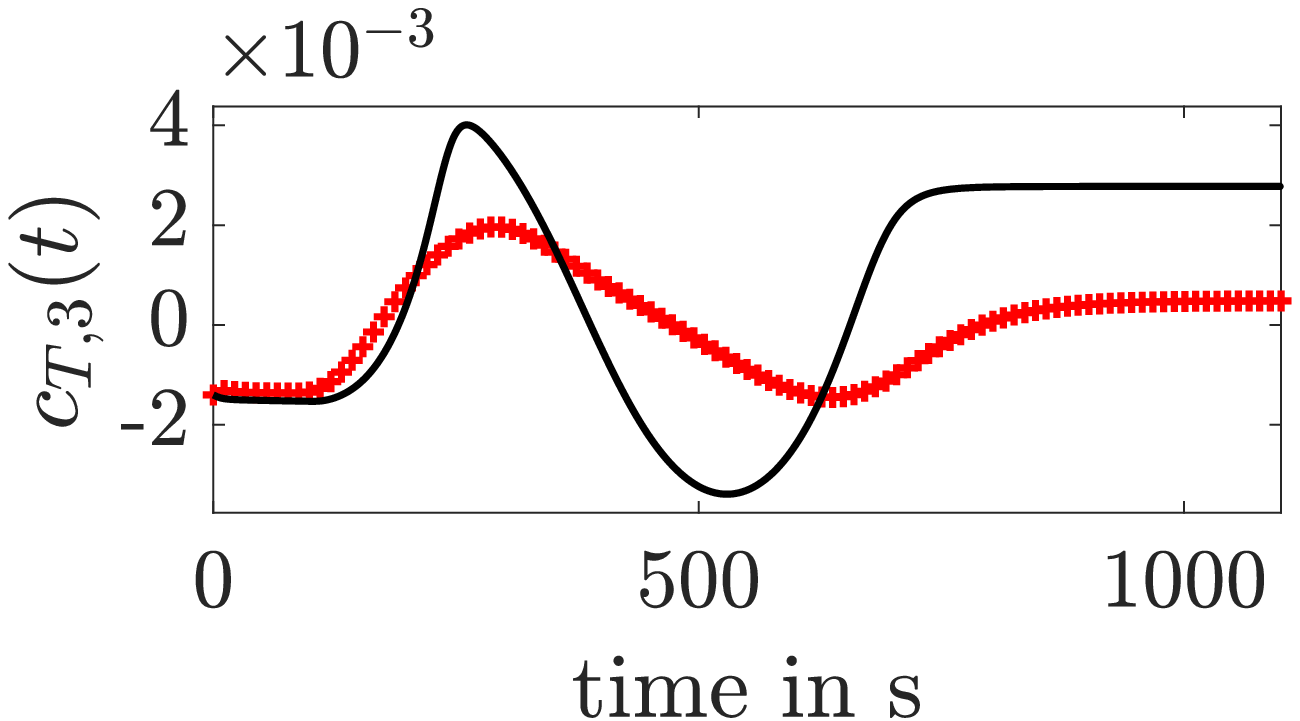}}
               \caption{Coefficients from approximation \eqref{eqn:SumLinApproxWater} (left) and \eqref{eqn:SumLinApprox} (right). The time-continuous results of the ROM with $n=6$ (solid black) are compared to time-discrete coefficients (red dots) from the original simulation of~\eqref{eqn:PDE}.}
        \label{fig:PODCoeff}
\end{figure}

We further check if the ROM represents the drying behavior of the wood chip reasonably well by analyzing the temporal and spatial behavior. Figure~\ref{fig:PODCoeff} compares the time-discrete coefficients obtained by the POD \eqref{eqn:SumLinApprox} and \eqref{eqn:SumLinApproxWater} (red crosses) to the time-continuous coefficients that result from solving the ODE system \eqref{eqn:ROMGauss} and the corresponding system for the moisture (black lines) for a step of the ambient temperature to $T_ \infty(t\geq 0)=\unit[373.15]{K}$. Some deviations occur for higher-order modes, but the most important modes match very well.
Furthermore, we determine the error
\begin{align}\begin{aligned}\label{eqn:ErrorDistribution}
x(y_i,t_j)&-\big(\bar{x}(y_i)  + \textstyle\sum_{k=1}^{n_x} \varphi_{x,k}(y_i) c_{x,k}(t_j)\big) \\
T(y_i,t_j)&-\big(\bar{T}(y_i)  + \textstyle\sum_{k=1}^{n_T} \varphi_{T,k}(y_i) c_{T,k}(t_j)\big).
\end{aligned}
\end{align}
to compare the spatial error of the moisture and temperature distribution inside the wood chip.
The maximum absolute error over all times and spatial locations is $\unit[24.3]{K}$ at time $t_j=\unit[704]{s}$ for the temperature and $\unit[0.094]{\nicefrac{\text{kg}}{kg}}$ at time $t_j=\unit[550]{s}$ for the moisture. The NRMSE~\eqref{eqn:MeanRelError} for the temperature and moisture distribution are $\varepsilon_T(n=6)=5.6\%$ and $\varepsilon_x(n=6)=2.5\%$, respectively. We repeated the analysis for all impulse responses used in section~\ref{subsec:CTRLDryingProcess} (as opposed to the step responses discussed in the present section). The NRMSE for temperature and moisture amount to about $0.5\%$ and $1.3\%$ in these cases.

\begin{figure}[t]
        \centering
\subfigure{\includegraphics[width=0.35\textwidth]{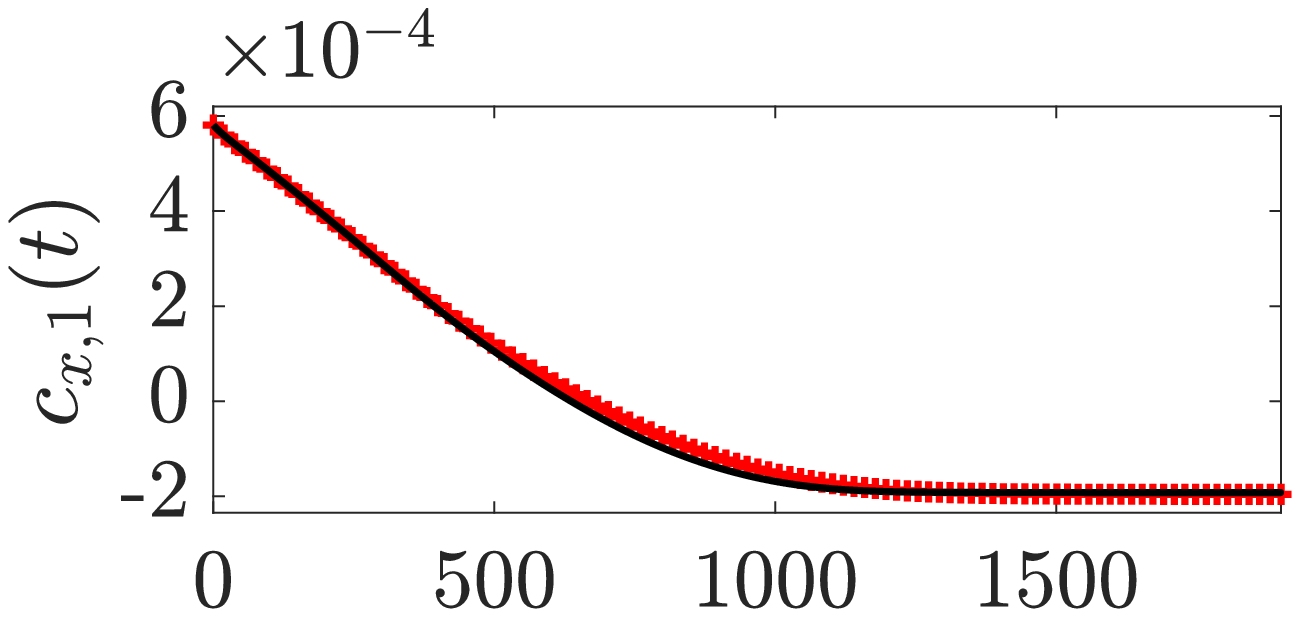}}
\subfigure{\includegraphics[width=0.35\textwidth]{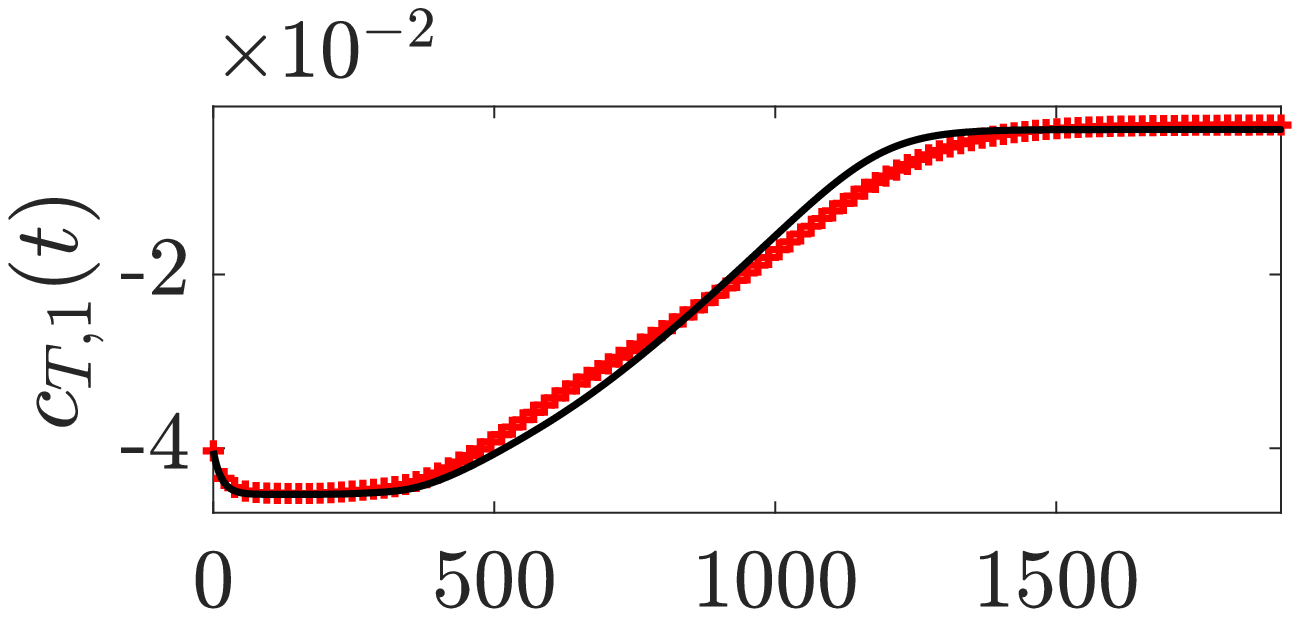}}\\
\subfigure{\includegraphics[width=0.35\textwidth]{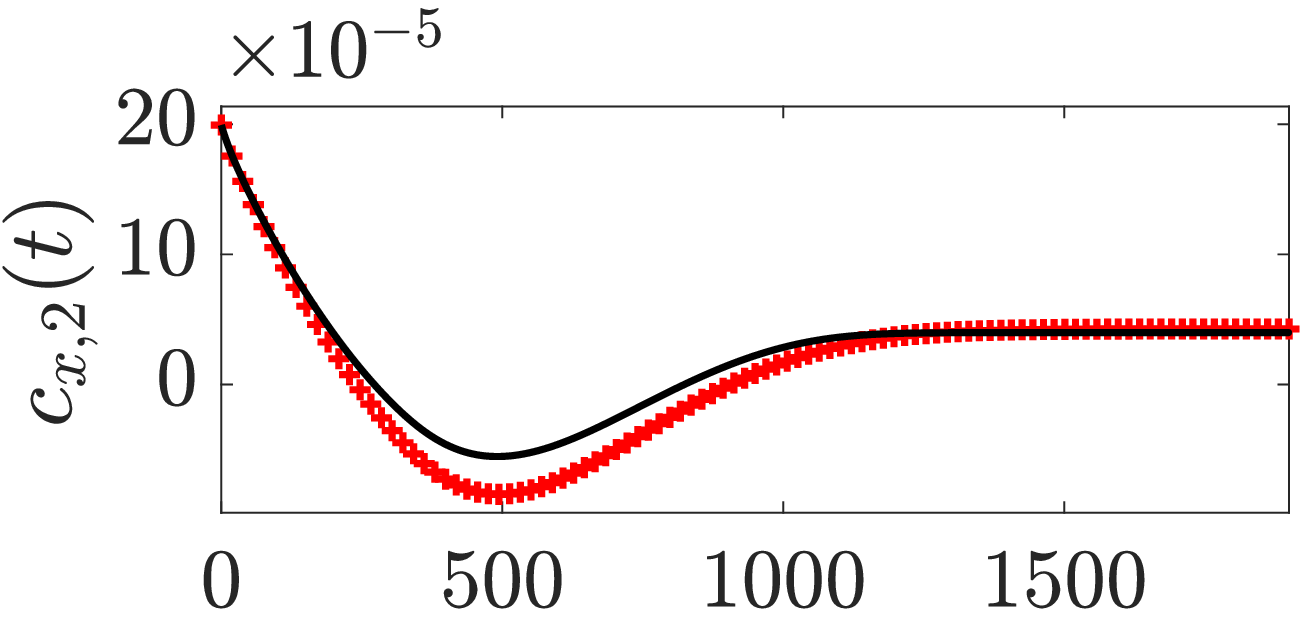}}
\subfigure{\includegraphics[width=0.35\textwidth]{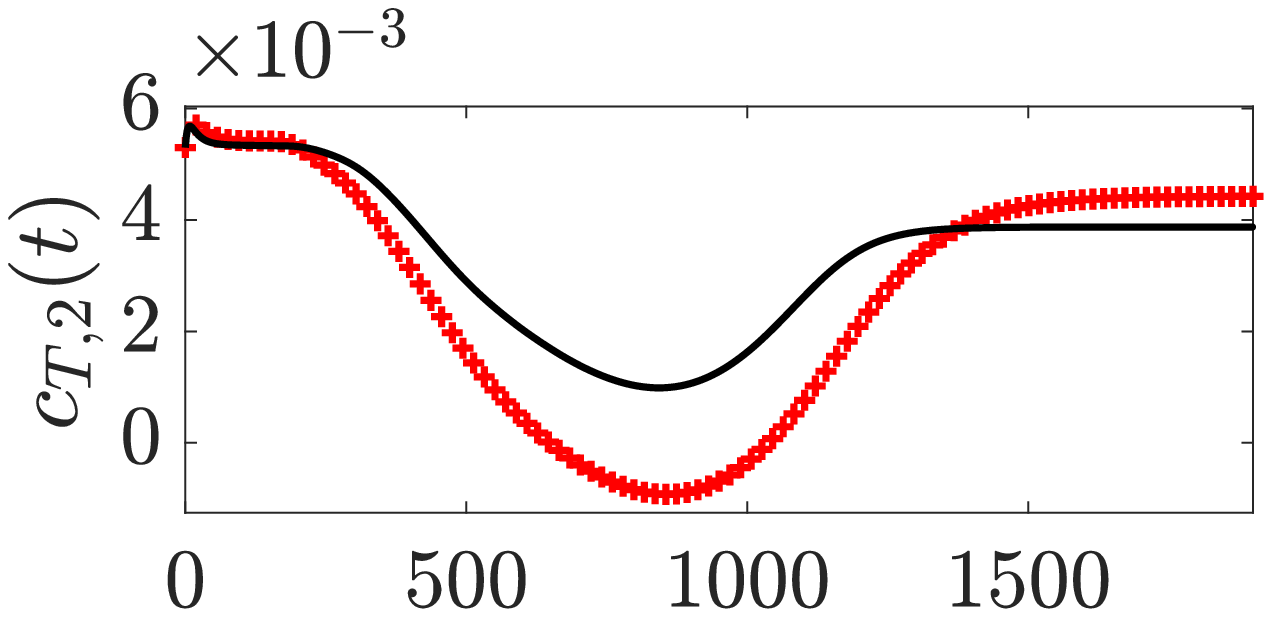}}\\
\subfigure{\includegraphics[width=0.35\textwidth]{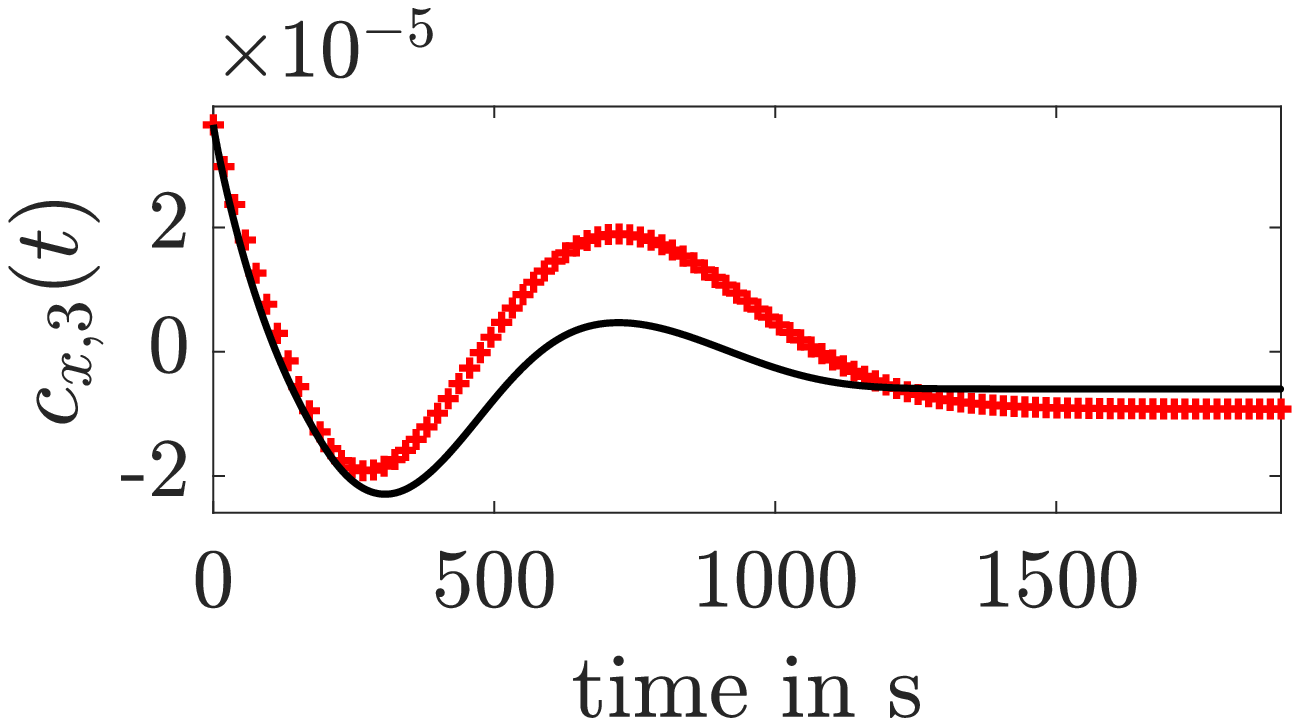}}
\subfigure{\includegraphics[width=0.35\textwidth]{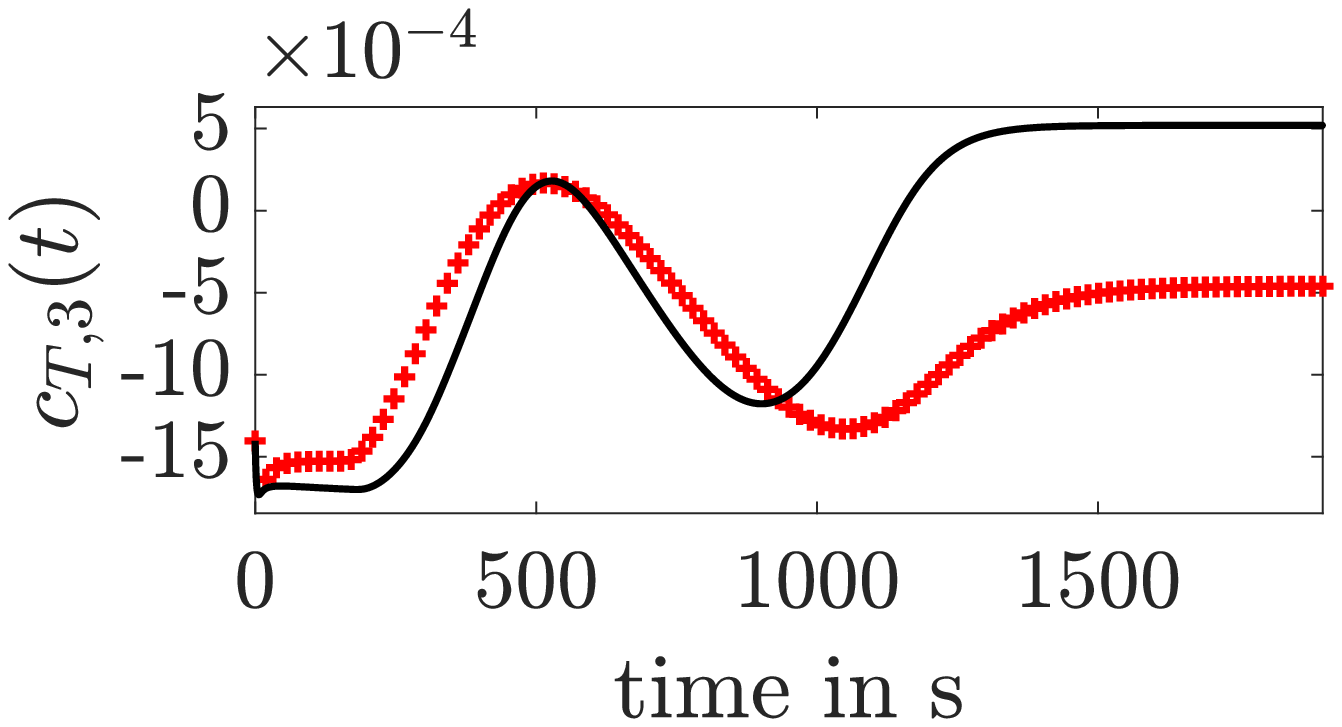}}
               \caption{The time-continuous results of the ROM for case B (solid black) are compared to time-discrete coefficients (red dots) that result from the original simulation of~\eqref{eqn:PDE} with a changed ambient temperature.}
        \label{fig:PODCoeffOtherBC}
\end{figure}

\begin{figure}[b]
        \centering
                       \includegraphics[width=0.7\textwidth]{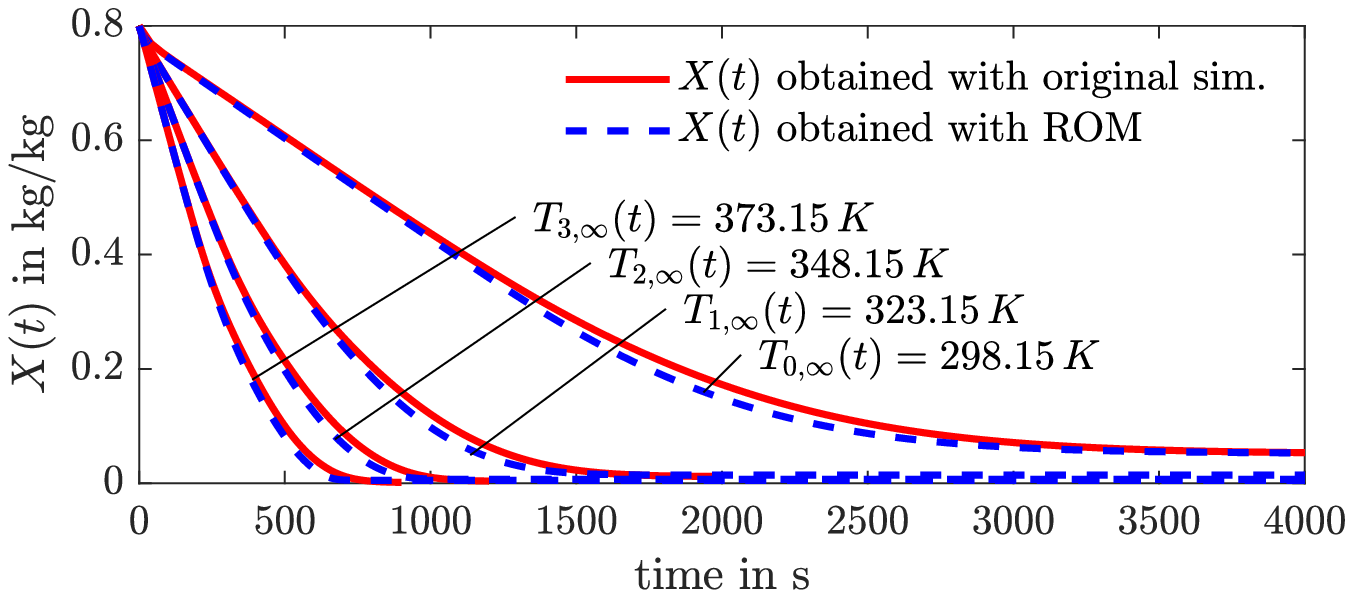}
               \caption{Total moisture $X(t)$ obtained with the ROM (dashed blue lines) and with the original simulation of~\eqref{eqn:PDE} (solid red lines) for the conditions stated in table~\ref{tb:CFDsimulations} and several ambient temperatures. We note for use in section~\ref{subsec:OCPWoodChip} that the maximum absolute error amounts to $\unit[2.4\times 10^{-2}]{kg/kg}$. It occurs for $T_\infty=\unit[298.15]{K}$ and $t=\unit[1920]{s}$.}
        \label{fig:TotalWaterCmp}
\end{figure}

We check if the ROM is also capable of representing the moisture and temperature inside a wood chip for significantly different than the design boundary conditions. 
This becomes crucial when the ROM is used in an optimization scheme where the ambient conditions are altered. 
We determine a ROM for the conditions stated in table \ref{tb:CFDsimulations} case A and apply the ambient temperature of case B. 
The time coefficients of the ROM are shown in Figure \ref{fig:PODCoeffOtherBC} (black lines). 
Just for comparison reasons we carry out a simulation of the full model~\eqref{eqn:PDE} for the conditions of case B and determine the time discrete coefficients (red dots). 
We observe that some deviations occur for higher order modes but the most important modes match acceptably well. 
We stress again that the full simulation for case B was not used to determine a ROM but only to determine the time coefficients for comparison reasons. 

Finally, we validate the ROM by comparing the total moisture $X(t)=\frac{1}{N}\textstyle\sum_{i=1}^N x(y_i,t)$ obtained with the ROM to the result of the original simulation for different step heights of the ambient temperature. Specifically, we choose $T_\infty(t\geq 0) \in \{ \unit[298.15]{K}, \unit[323.15]{K}, \unit[348.15]{K},$ $\unit[373.15]{K}\}$. We analyze the total moisture, because this quantity is required in the optimal control problem presented in section \ref{sec:OCP}. The approximation of the total moisture by the ROM is shown in Figure \ref{fig:TotalWaterCmp} (dashed blue lines). Minor deviations occur in the middle of the drying process. We claim that this approximation is sufficiently accurate for the use in an optimal control problem. 
Note that the ROM was determined only from simulation results for the full model~\eqref{eqn:PDE} for $T_\infty(t\geq 0) = \unit[373.15]{K}$. The simulation results for $T_\infty(t\geq 0) \in \{ \unit[298.15]{K}, \unit[323.15]{K}, \unit[348.15]{K}\}$ were only used for the validation. 

We briefly note that the approximation error of the ROM is not negligible but acceptable, since it has the same order of magnitude as the approximation error of the full model~\eqref{eqn:PDE} itself. In \cite{Sudbrock2014}, the drying behavior of a single sphere-shaped wood particle was determined experimentally and compared to simulations with the full model. While these results cannot be compared to the results obtained here due to the different particle geometry, a comparison of the approximation errors is still useful. 
The NRMSE between the simulations and experimental results amounts to $6.3\%$ for the drying rate. In comparison, the ROM of order $n=6$ in section 4.1 results in a NRMSE of $3.1\%$ with respect to the original simulation data.
We conclude the ROM represents the wood chip drying process sufficiently accurately, since the error due to the model reduction is smaller than the modeling error.

\subsection{Controllability of the drying process}\label{subsec:CTRLDryingProcess}
We apply the empirical controllability Gramian as introduced in sections~\ref{subsec:CtrbGramROM} to the ROM of section~\ref{subsec:modelreduction}. 
Specifically, we apply the control input \eqref{eqn:GramImpulseInput} with $h_d\in\{ 10^{-3},10^{-2},10^{-1}, 10^{0},\\10^{1},10^{2},10^{3}\}$, i.e., $s= 7$,  
to the ROM from section~\ref{subsec:ROMevaluation} in order to approximate the controllability Gramian (11) by $W$ according to proposition~\ref{Prp:ApproxGram}. 
The ROM is initialized in the steady state for $u_0=\unit[298.15]{K}$. 
We choose the values for $h_d$ listed above to cover $7$ orders of magnitude. 
The remaining parameters of the control input function~\eqref{eqn:GramImpulseInput} read 
$D_l = 1$, $l= 1$ and $e_i=1$, $i=1$, since only $\gamma=1$ input exists in this case.
Solving the ROM \eqref{eqn:ROMGauss} for each $h_d$ yields the desired time-series $c^{dli}(t)$ for the state variables~\eqref{eqn:PODGalStates} of the ROM and their steady states $c^{dli}_{\text{ss}}$. We write $c^d(t)$ short for $c^{d11}(t)$ and $z_{d}(t)$ for $z_{d11}(t)$ below.

We use the coefficients $c^d(t)$ to determine 
\begin{align}\label{eqn:CtrbGramROMDiscrete}
W^*= \sum_{d=1}^s\frac{1}{sh_d^2}\sum_{j=0}^{m_\text{f}}\big(c^d(t_j)-c^d(t_{m_\text{f}}) \big)\big(c^d(t_j)-c^d(t_{m_\text{f}}) \big)^\T   \Delta t,
\end{align}
i.e., the discrete-time representation of \eqref{eqn:CtrbGramROM} where the integral in \eqref{eqn:CtrbGramROM} is approximated by a sum with $m_\text{f}=15 \cdot 10^6$ time steps of step size $\Delta t=\unit[0.001]{s}$. The parameters $m_\text{f}$ and $\Delta t$ were chosen such that an increase of one order of magnitude of the discretization time results in a change of less than $1\%$ for \eqref{eqn:CtrbGramROMDiscrete} and such that $c^d(t_{m_\text{f}})$ is the steady state $c^d_\text{ss}$. 
We have $W^*\in\R^{6 \times 6}$, since the ROM is of order $n=6$.

\begin{figure}[t]
        \centering
\subfigure{\includegraphics[width=0.35\textwidth]{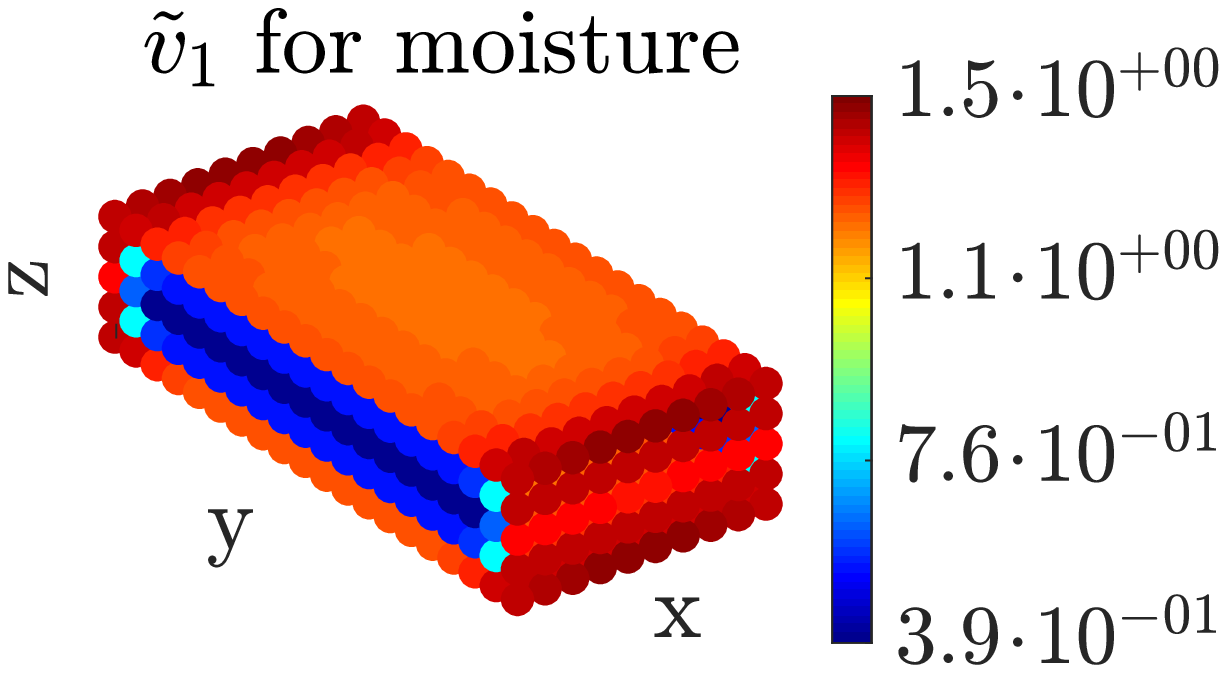}}
\subfigure{\includegraphics[width=0.35\textwidth]{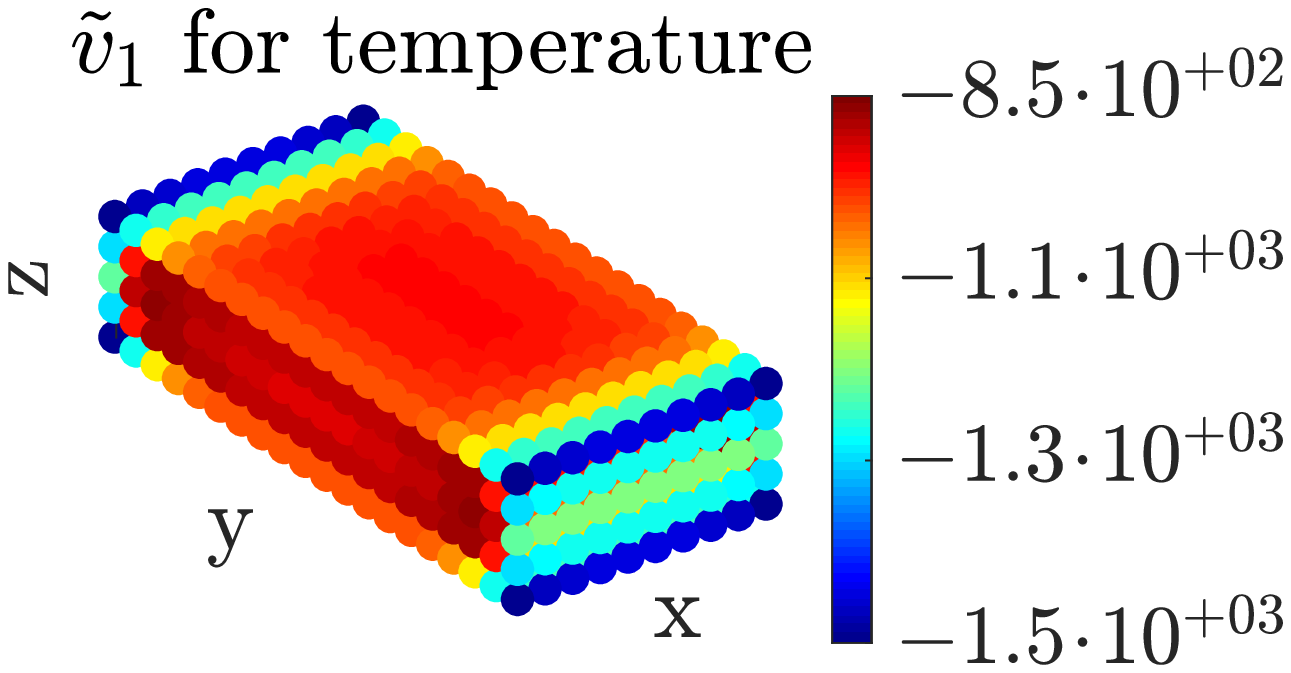}}
        \caption{Approximation of the first eigenvector $\tilde{v}_{1}$ determined with a ROM of order $n=6$.} 
        \label{fig:MostControllableModes}
\end{figure}

\begin{table}[b]
\begin{center}
\caption{Eigenvalues of the empirical controllability Gramian for a ROM of order $n=6$}\label{tb:EigVal}
\begin{tabular}{rcl}
$\tilde{\beta}_1$ & $=$ & $6.91$ \\
$\tilde{\beta}_2 $ & $=$ & $  2.06 \times 10^{-1}$ \\
$\tilde{\beta}_3 $ & $=$ & $  5.72 \times 10^{-3}$ \\
$\tilde{\beta}_4 $ & $=$ & $  9.34 \times 10^{-6}$ \\
$\tilde{\beta}_5 $ & $=$ & $  1.54 \times 10^{-6}$ \\
$\tilde{\beta}_6 $ & $=$ & $  3.16 \times 10^{-8}$
\end{tabular}
\end{center}
\end{table}

The eigenvalues $\tilde{\beta}_k$, $k=1,\ldots,6$ are determined with~\eqref{eqn:EigProblemSmall} and listed in table \ref{tb:EigVal}. We find that $\tilde{\beta}_k > 0$ for all $k=1,\ldots,6$. According to Lemma \ref{lma:LinearCtrbCond}, this indicates that the nonlinear ROM of section \ref{subsec:modelreduction} is controllable. We conclude that the control input $u(t)=T_\infty(t)$ is a reasonable choice to control the states of the ROM. However, we cannot infer the $M$-dimensional finite-volume model~\eqref{eqn:NonlinSysAllg} to be controllable or not, since $\tilde{\beta}_k$ are the approximations for only some eigenvalues of the larger controllability Gramian \eqref{eqn:CCM}. It is possible that \eqref{eqn:CCM} has zero eigenvalues and \eqref{eqn:CtrbGramROMDiscrete} has not. In fact, we expect that the detailed model is not fully controllable, since the wood chip drying problem, the particle volume $\Omega$, the boundary conditions and the spatially dependent material parameters are symmetric. Due to this symmetry, arbitrary moisture and temperature distributions are not possible. However, we can determine the controllable subspace according to Lemma \ref{lma:CtrlSubspace} using the eigenvectors of \eqref{eqn:CtrbGramROMDiscrete}. The eigenvectors \eqref{eqn:LinApproxEig} approximate the controllable subspace of the large model. The eigenvector $\tilde{v}_1$ indicating the most controllable direction is shown in Figure \ref{fig:MostControllableModes} for illustration.

We claim the ROM is suitable for controlling the moisture and temperature distribution, since the ROM is controllable and its reachable states yield an approximation for the reachable moisture and temperature distribution of the detailed model \eqref{eqn:PDEdiscrete}.

\subsection{Controllability comparison of different reduced models}\label{subsec:ComparisonOfROMs}
As a final preparation, it remains to check if the controllability properties change when the order of the ROM is changed. Specifically, we check if the eigenvalues $\tilde{\beta}_k$ 
change for ROM of different order by repeating the analysis performed for $n=6$ in section~\ref{subsec:CTRLDryingProcess} for $n=6, 8, \dots, 50$, where $n=50$ is an arbitrary high number.
Figure~\ref{fig:SingularValuesGram} shows the eigenvalues $\tilde{\beta}_{k,n}$, where the subscript $k, n$ refers to the $k$-th eigenvalue of the Gramian $W_n^*$. 
The eigenvalues appear in pairs and the smallest eigenvalue pairs decrease with increasing order of the ROM. All other eigenvalues remain nearly unchanged.
The new eigenvalues that appear when increasing the order from $n= 6$ to $n= 8$ are smaller than the leading ones by about four orders of magnitude. 
Consequently, the controllability properties already established in section~\ref{subsec:CTRLDryingProcess} for $n= 6$ do not improve for increased orders. 
Since $n= 6$ also proved to result in a sufficiently precise model in section~\ref{subsec:ROMevaluation}, we choose $n= 6$ for the ROM used in the optimal control problems.

\begin{figure}[t]
        \centering
        \includegraphics[width=0.7\textwidth]{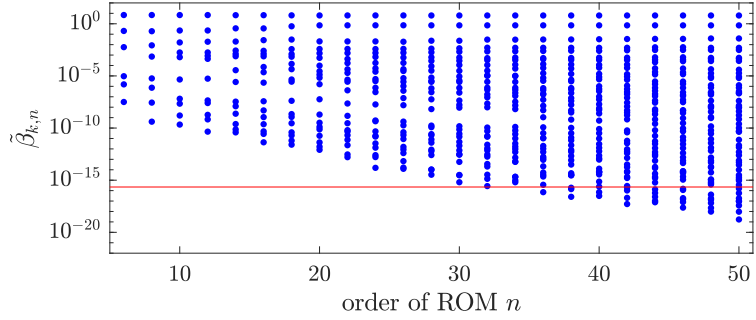}
               \caption{Eigenvalues $\tilde{\beta}_{k,n}$, $k=1,\ldots,n$ of the empirical controllability Gramian $W_n^*$ for ROM orders $n=6, 10, \ldots,50$ (blue dots). The horizontal red line marks the floating-point accuracy. }
        \label{fig:SingularValuesGram}
\end{figure}

\section{Optimal Control Problem}\label{sec:OCP}

\subsection{Optimal control problem under consideration}\label{sec:consideredOCP}

We assume the ambient temperature $T_\infty(t)=u(t)$ to be the only control input. Furthermore, we assume that $u(t)$ is subject to bounds 
\begin{align}\label{eqn:inputconstraints}
  u_\text{min}<u(t)<u_ \text{max} \mbox{ for } t\in [0, t_\text{f}],
\end{align}
where the process starts at $t= 0$ without restriction, and where $t_\text{f}$ is a given end time of the drying process. 
It is our goal to find a control trajectory so that the total moisture \eqref{eqn:overallwatercontent} in the wood particle is less than a prescribed value $X_\text{f}$ at the end of the drying process. This is enforced by the terminal inequality constraint
\begin{align}
\label{eqn:terminalwatercontent}
X(t_\text{f}) \leq X_\text{f}.
\end{align} 
The cost function
\begin{align}
\label{eqn:costfunction}
J(u(\cdot))=\int_0^{t_\text{f}}u(t)- \unit[298.15]{K}\;\text{d}t
\end{align} 
serves as a simple model for the cost of energy. 

In summary, we seek the function $u:[0, t_\text{f}]\rightarrow\R$
that minimizes~\eqref{eqn:costfunction} subject to the input constraints~\eqref{eqn:inputconstraints}, the terminal constraint~\eqref{eqn:terminalwatercontent} for the integral moisture~\eqref{eqn:overallwatercontent}, and 
the dynamics~\eqref{eqn:PDE} with boundary and initial conditions \eqref{eqn:PDEBC} and \eqref{eqn:PDEInitCond}, respectively, where $t_\text{f}$ is a given end time. 

Since we cannot expect to find an analytic solution, the stated optimal control problem must be solved numerically. However, solving the OCP numerically with an embedded solver for the original model~\eqref{eqn:PDE} is tedious and computationally expensive. For this reason, the ROM presented in section~\ref{subsec:modelreduction} is used to approximate the PDEs in the optimal control problem stated above. This substitution results in the optimal control problem

\begin{equation}
\begin{aligned}
\label{eqn:ROMOptimizationProblem}
\min_{u(t_j),\, j=0, \dots, m} &\sum_{j=0}^{m}\big( u(t_j)-\unit[298.15]{K}\big)\Delta t\\
\text{subject to} & &&\\
\dot{c}_{x,k}(t) &= f_\text{ROM,x}\big(c_{x,k}(t),c_{T,l}(t)\big)\\
\dot{c}_{T,l}(t) &= f_\text{ROM,T}\big(c_{x,k}(t),c_{T,l}(t),u(t)\big)\\
c_{x,k}(t=0)&=c_{x,k}(t_0),\\
c_{T,l}(t=0)&=c_{T,l}(t_0),\\
x(y_i,t_\text{f}) &= \bar{x}(y_i)+ \textstyle\sum_{k=1}^{n_x} \varphi_{x,k}(y_i) c_{x,k}(t_\text{f})\\
X(t_\text{f})&=\frac{1}{N}\sum_{i=1}^{N} x(y_i,t_\text{f})\\
X(t_\text{f}) &< X_\text{f}\\
u_\text{min}&<u(t_j)<u_\text{max}
\end{aligned}
\end{equation}
with $k=1,\ldots,n_x$, $l=1,\ldots,n_T$, $i=1,\ldots,N$, $j=0,\ldots,m$ and where $f_\text{ROM,x}(c_{x,k}(t),c_{T,l}(t))$ and $f_\text{ROM,T}(c_{x,k}(t),c_{T,l}(t),u(t))$ refer to the right hand side of \eqref{eqn:ROMGauss} and its moisture equivalent. 
The input function $u(t)$ is discretized with zero-order hold and a step size of $\unit[1]{s}$, where $u(t_j)$, $j= 0, \dots, m$ with $m= 600$ steps will be required in section~\ref{subsec:OCPWoodChip}. 
The integral in \eqref{eqn:overallwatercontent} is approximated by a sum and the ODEs are solved with an explicit Euler integration with step size \unit[1]{s}.

\subsection{Optimal control results for the drying of wood chips}\label{subsec:OCPWoodChip}

We determine the optimal input sequence $u(t_j)$, $j= 0, \dots, m$, with $m=600$, for the drying process with 
a target moisture of $X_\text{f}=\unit[1\times 10^{-1}]{\nicefrac{\text{kg}}{kg}}$ and $t_\text{f}=\unit[600]{s}$. 
The bounds on the input read $u_ \text{min}=\unit[298.15]{K}$ and $u_\text{max}=\unit[373.15]{K}$. 
We choose $n_x=n_T=3$, thus $n=6$, for the order of the ROM. 
We use an interior-point algorithm to solve the resulting finite-dimensional optimization problem.\footnote{Matlab's \textit{fmincon} required $\unit[12858]{s}$ on an i7-6700 CPU at 3.40GHz.}
Since this algorithm is not guaranteed to find the global minimum, but in general terminates at a local minimum,
we solved the optimal control problem for $5$ constant temperature profiles with $T\in [\unit[298.15]{K}, \unit[373.15]{K}]$. The same optimal solution resulted in all cases.

The solution to the optimal control problem~\eqref{eqn:ROMOptimizationProblem} 
is shown in Figure~\ref{fig:OCPoptimalTrajectory} (red line). 
It turns out to be a bang-bang solution with two heating and two resting periods. The control attains the upper bound during the heating periods $0\leq t \leq \unit[219]{s}$ and $\unit[390]{s} \leq t \leq \unit[591]{s}$ and the lower bound during the resting periods $\unit[219]{s} < t< \unit[390]{s}$ and $\unit[591]{s} < t < \unit[600]{s}$. The heating periods are located at the very beginning and almost at the end of the drying process. 

Bang-bang solutions are known to be optimal for simple cost functions like \eqref{eqn:costfunction} \cite[Ch. 7.4]{Naidu2002}.
Despite the simplicity of the cost function, the optimal control problem reveals it to be attractive not to heat the particles constantly. This result is physically meaningful, which can be seen as follows. 
Since evaporation takes place on the surface only, drying is faster on the particle surface and slower inside the particle. At some point during the drying process, the inner particle is still wet but the surface is already dry so that the evaporation rate drops and drying proceeds slowly. Keeping the ambient temperature low during this time saves energy and allows the moisture inside the particle to diffuse to the surface. Evaporation increases in the subsequent heating period and drying proceeds faster.

\begin{figure}[t]
        \centering
        \includegraphics[width=0.7\textwidth]{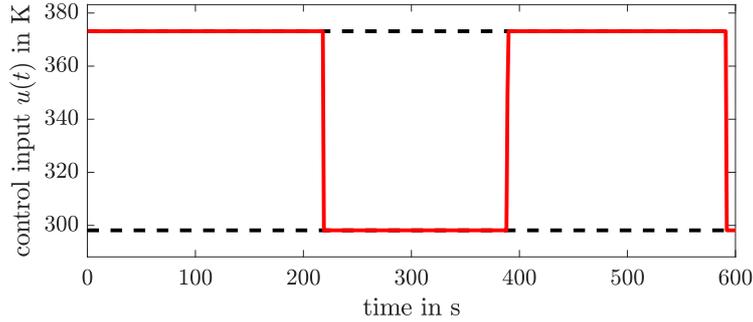}
               \caption{Optimal control trajectory obtained with a ROM of order $n=6$ (solid red line). The upper and lower bounds are $u_ \text{min}=\unit[298.15]{K}$ and $u_\text{max}=\unit[373.15]{K}$ (dashed black lines).  }
                       \label{fig:OCPoptimalTrajectory}
\end{figure}

The trajectory that results for the total moisture with the optimal $u(t)$ is shown in Figure~\ref{fig:TotalWaterCmpOCP} (solid red line). The Figure shows the result predicted by the ROM for comparison (dashed blue line). More specifically, the dashed blue line in Figure~\ref{fig:TotalWaterCmpOCP} is the moisture that results from integrating the ODEs for $c_{x,k}(t)$ in~\eqref{eqn:ROMOptimizationProblem} and determining $x(y_i, t)= \bar{x}(y_i)+\textstyle\sum_{k=1}^{n_x}\varphi_{x,k}(y_i) c_{x,k}(t)$ and $X(t)= \frac{1}{N}\sum_{i=0}^N x(y_i, t)$. 
For both the simulation with the PDEs and the ROM, the total moisture decreases from an initial value of about $X(t=0)=\unit[8\times 10^{-1}]{\nicefrac{\text{kg}}{kg}}$ and attains the desired target value of $X_\text{f}=\unit[1\times 10^{-1}]{\nicefrac{\text{kg}}{kg}}$ (marked by the dash-dotted black line) at $t_\text{f}=\unit[600]{s}$. As expected, 
the total moisture decreases faster during the heating periods and more slowly during the resting period. 
The ROM reaches the target value earlier than the original simulation. The absolute error in Figure~\ref{fig:TotalWaterCmpOCP} 
at $t=\unit[600]{s}$ amounts to $\unit[2.3\times 10^{-2}]{kg/kg}$ and thus is as large as the maximum absolute error for $X(t)$ found in section~\ref{subsec:ROMevaluation} (cf.\ Figure~\ref{fig:TotalWaterCmp}). 
Since the latter maximum absolute error is within the approximation precision of the original PDEs (cf.\ the last paragraph of section~\ref{subsec:ROMevaluation}), we consider the deviation of the trajectories at $t=\unit[600]{s}$ in Figure~\ref{fig:TotalWaterCmpOCP} to be acceptable. 
Note that we compare maximum absolute errors here as opposed to NRMSEs in section~\ref{subsec:ROMevaluation}, since we are interested in the maximum error in time here. 

\begin{figure}[t]
        \centering
                       \includegraphics[width=0.7\textwidth]{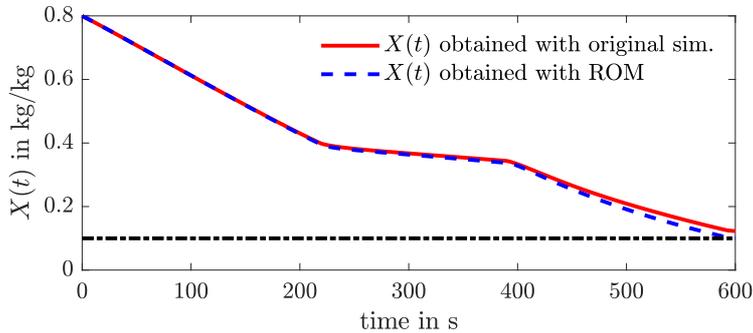}
               \caption{Total moisture $X(t)$ for the optimal drying of wood chips obtained by applying the optimal input sequence to the ROM ($n= 6$, dashed blue line) and to the simulation of the finite-volume model~\eqref{eqn:NonlinSysAllg} (solid red line). Small deviations occur at the end of the drying process. The target value $X_\mathrm{f}=1\times 10^{-1}\nicefrac{\mathrm{kg}}{\mathrm{kg}}$ that is enforced by the terminal constraint is marked by the dashed line.}
        \label{fig:TotalWaterCmpOCP}
\end{figure}

\subsection{Reduced order model study}\label{subsec:OCPstudyDifferentROM}

Choosing the number of modes $n$ obviously involves a trade-off between the degree of reduction and the approximation accuracy. 
We analyze this trade-off by comparing optimal control results obtained from ROM of orders $n= 6, 10, 34$. We choose $n=6$, since it results in the smallest ROM with acceptable approximations for temperature and moisture, and $n=34$, since it is the largest controllable ROM according to section \ref{subsec:ComparisonOfROMs}. The order $n= 10$ is an arbitrary intermediate value. 

We solve the optimal control problem for $n=6,\,10,\,34$ with the same conditions as stated in section \ref{subsec:OCPWoodChip}. 
Computation times and cost function values are listed in table \ref{tb:CompTimeComparison}.
The optimal controls and the resulting total moistures are shown in Figure~\ref{fig:OCPTrajecCmpDiffOrdr}.  
All optimal controls are of bang-bang type with two heating and two resting periods. The switching points nearly coincide for all $n$. 
The total moistures that result from applying the optimal controls to the finite-volume model~\eqref{eqn:NonlinSysAllg}, which are shown in Figure~\ref{fig:TotalWaterCmpDiffOrdr}, nearly coincide. While the deviations at the end of the drying process get smaller as the ROM order is increased, the deviation for $n= 6$ is already acceptable as discussed at the end of section~\ref{subsec:OCPWoodChip}. 

\begin{figure}[t]
        \centering
                       \includegraphics[width=0.7\textwidth]{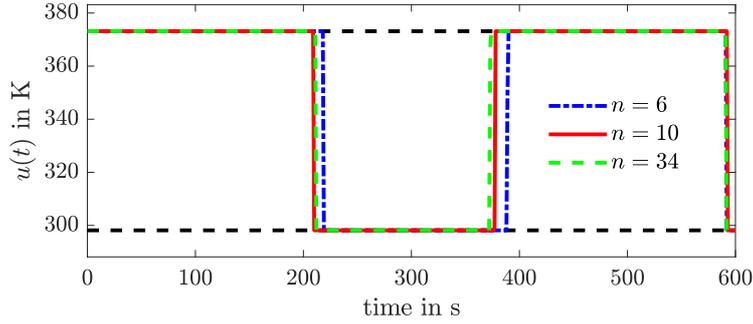}
               \caption{Optimal control trajectories obtained with ROM of orders $n=6, 10, 34$.}
        \label{fig:OCPTrajecCmpDiffOrdr}
\end{figure}

\begin{figure}[t]
        \centering
                       \includegraphics[width=0.7\textwidth]{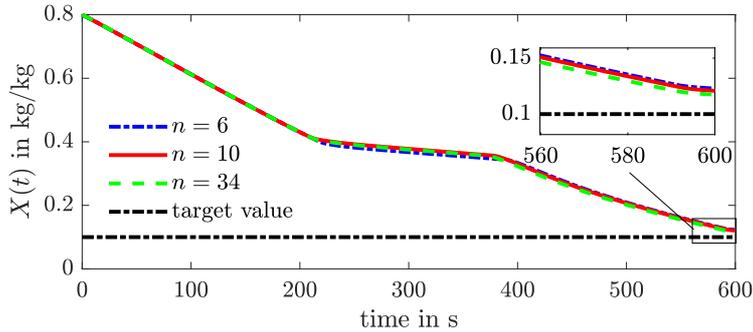}
               \caption{Total moisture $X(t)$ obtained from simulations of the original model~\eqref{eqn:PDE} with the optimal input sequence obtained with the ROM with $n=6, 10, 34$.}
        \label{fig:TotalWaterCmpDiffOrdr}
\end{figure}

\begin{table}[b]
\begin{center}
\caption{Computation times for the optimal control problem for ROM of different order}\label{tb:CompTimeComparison}
\begin{tabular}{lll}
ROM order & computation time & cost function values\\ \hline
$n=6$ & $\unit[12858]{s}$ & $J=31627$ \\
$n=10$ & $\unit[16316]{s}$ & $J=31875$\\
$n=34$ & $\unit[36936]{s}$ & $J=32319$
\end{tabular}
\end{center}
\end{table}

\section{Conclusion}\label{sec:Outlook}
We used POD and Galerkin-based model reduction to obtain a ROM for the drying of wood chips. Specifically, a ROM of order six proved to be appropriate to approximate the coupled heat and moisture diffusion. We used the model for a nonlinear controllability analysis of the drying process. 
The eigenvalues of the empirical controllability Gramian were used as a controllability measure. We showed that the ROM of order six is controllable and that its states yield a reasonable approximation of the controllable subspace of the drying process. Furthermore, the model proved to be sufficiently accurate and computationally efficient to allow solving optimal control problems for the energy-efficient operation of the drying process. We demonstrated new modes of operation for drying processes can easily be explored with optimal control problems, once a ROM is available.

\appendix

\section*{Appendix A} \label{sec:Apx:WoodChipModel}
The volumetric heat capacity $s$ and diffusion coefficients $\lambda$ and $\delta$ in the PDEs \eqref{eqn:PDE} depend on the local temperature $T$ and moisture $x$. They read
\begin{align*}
s(x) &= \rho_\text{d}\big(1+x\big) \frac{c_\mathrm{p,d}+x\,c_\text{p,w}}{1+x}\\
\lambda(x) &= \mathrm{diag}\big( \lambda_\mathrm{x}(x),\lambda_\mathrm{y}(x),\lambda_\mathrm{z}(x)\big)\\
\delta(T) &= \mathrm{diag}\big( \delta_\mathrm{x}(T) , \delta_\mathrm{y}(T) ,  \delta_\mathrm{z}(T) \big),
\end{align*}
where $\lambda(x)\in\R^{3 \times 3}$ and $\delta(T)\in\R^{3 \times 3}$ due to the anisotropy of the wood,
\begin{align*}
\lambda_i(x)&= \lambda_{\text{d},i}+\frac{x\,\lambda_\text{w} }{1+x} \\
\delta_i(T)&= \delta_{\text{d},i}\bigg(\frac{T}{293.15}\bigg)^{1.75}
\end{align*}
for $i\in\{\mathrm{x},\mathrm{y},\mathrm{z}\}$ 
\cite[Table 3]{Scherer2016}, and where all constants can be found in Table \ref{tb:AppxConstParam}.
\begin{table}[t]
\begin{center}
\caption{Parameters and conditions for the drying process of wood chips}\label{tb:AppxConstParam}
\begin{tabular}{lrl}
wood chip volume & $V$&$= \unit[1\cdot 10^{-6}]{m^3}$\\ 
density of dry wood & $\rho_\text{d} $&$= \unit[500]{\nicefrac{kg}{m^3}}$\\ 
heat capacity of dry wood & $c_\mathrm{p,d} $&$= \unit[1500]{\nicefrac{J}{kg\,K}} $\\ 
heat capacity of water & $c_\mathrm{p,w} $&$= \unit[4190]{\nicefrac{J}{kg\,K}} $\\ 
thermal conductivity of water & $\lambda_\mathrm{w} $&$= \unit[0.56]{\nicefrac{W}{m\,K}} $\\ 
thermal conductivity in fiber direction & $\lambda_{\mathrm{d,x}} $&$= \unit[1\cdot 10^{-7}]{\nicefrac{W}{m\,K}} $\\ 
thermal conductivity orth.\ to fiber direction & $\lambda_{\mathrm{d,y}} $&$=\lambda_{\mathrm{d,z}}  = \unit[2\cdot 10^{-9}]{\nicefrac{W}{m\,K}} $\\ 
mass diffusion coefficient in fiber direction  & $\delta_{\mathrm{d,x}} $&$= \unit[0.24]{\nicefrac{m^2}{s}} $\\ 
mass diffusion coefficient orth.\ to fiber direction & $\delta_{\mathrm{d,y}}$&$=\delta_{\mathrm{d,z}}= \unit[0.12]{\nicefrac{m^2}{s}} $\\ 
heat transfer coefficient &$\alpha$&$= \unit[45]{\nicefrac{W}{m^2\,K}} $\\
mass transfer coefficient&$\beta$&$= \unit[0.075]{\nicefrac{m}{s}} $\\
molar mass&$M_\mathrm{H_2O}$&$= \unit[18.01528\cdot 10^{-3}]{\nicefrac{kg}{mol}} $\\
gas constant&$R$&$= \unit[8.3144621]{\nicefrac{J}{mol\,K}} $\\
ambient humidity&$\rho_\infty$&$= \unit[0.007]{\nicefrac{kg}{m^3}} $
\end{tabular}
\end{center}
\end{table} 
The boundary conditions \eqref{eqn:PDEBC} depend on the absolute humidity on the surface \cite[eq. 3.43]{Sudbrock2014} 
\begin{align}
\rho(x,T) = M_\mathrm{H_2O}\frac{\varphi_\mathrm{s}(x,T) \cdot p_\text{v,sat}(T)}{R\cdot T},\label{eqn:IntroPDEAbsHumidity}
\end{align}
and the enthalpy of adsorption \cite[eq. 3.159]{Sudbrock2014}
\begin{align}
\Delta h_\text{ads}(x,T) = \Delta h_\text{v}(T) + \Delta h_\text{b}(x,T).\label{eqn:IntroPDEEnthalpyAds}
\end{align}
Due to the dependence on the local surface temperature $T$ and moisture $x$, \eqref{eqn:IntroPDEAbsHumidity} and \eqref{eqn:IntroPDEEnthalpyAds} are functions of location and time. 
The relative humidity $\varphi_\mathrm{s}$ (cf. \cite[eq.\,3.170]{Sudbrock2014}), the saturation vapor pressure $p_\text{v,sat}$ (cf. \cite{Buck1981,Kaempfer2012}), the evaporation enthalpy $\Delta h_\mathrm{v}$ (cf. \cite[eq.\,3.160]{Sudbrock2014}) and the bond enthalpy $\Delta h_\mathrm{b}$ (cf. \cite[eq.\,3.171]{Sudbrock2014}) read
\begin{align}
\varphi_\mathrm{s}(x,T) &= \left\{
    \begin{array}{ll}
      1-\Big( 1- \tfrac{x}{x_\text{fsp}(T)} \Big) ^{6.453 \cdot 10^{-3} \cdot T} & \mbox{for $x\leq x_\text{fsp}$}
      \\
      1 & \mbox{for $x>x_\text{fsp}$}.
    \end{array}
  \right.   \label{eqn:AppxRelHumidity}\\
p_\text{v,sat}(T) &= 611.21 \cdot \exp{ \Big( \big( 18.678-\tfrac{T - 273.15}{234.5} \big)  \tfrac{T-273.15}{T -16.01} \Big)}\label{eqn:AppxPvsat}\\
\Delta h_\text{v}(T) &= 3.1671\cdot 10^6 - 2433.2 \cdot T \label{eqn:AppxHv}\\
\Delta h_\text{b}(x,T) &= \left\{
    \begin{array}{ll}
      0.4 \cdot \Delta h_\text{v}(T) \cdot \Big( 1- \frac{x}{x_\mathrm{fsp}(T)} \Big)^2 & \mbox{for $x\leq x_\text{fsp}$}
      \\
      0 & \mbox{for $x>x_\text{fsp}$}
    \end{array}
  \right. \label{eqn:AppxHb}
\end{align} 
with the moisture at the fiber saturation point 
\begin{align}\label{eqn:AppxXfsp}
x_\text{fsp}(T) &= 0.598-0.001 \, T
\end{align}
\cite[eq. 3.169]{Sudbrock2014}. Since \eqref{eqn:AppxRelHumidity}-\eqref{eqn:AppxXfsp} are empirical functions, it remains to state their units. We have $[\varphi_\mathrm{s}]=1$, $[p_\text{v,sat}]=\mathrm{Pa}$, $[\Delta h_\mathrm{v} ]=\nicefrac{\mathrm{J}}{\mathrm{kg}}$, $[\Delta h_\mathrm{b}]=\nicefrac{\mathrm{J}}{\mathrm{kg}}$ and $[x_\text{fsp}]=\nicefrac{\mathrm{kg}}{\mathrm{kg}}$. 

\bibliographystyle{plain} 
\bibliography{BernerO2019}         

\end{document}